\documentclass[hidelinks,onefignum,onetabnum]{siamart220329}

\usepackage{lipsum}
\usepackage{amsfonts}
\usepackage{bm}
\usepackage{graphicx}
\usepackage{epstopdf}
\usepackage{algorithm}
\usepackage{algorithmicx}
\usepackage{algpseudocode}

\usepackage{multirow}
\usepackage{enumerate}
\ifpdf
  \DeclareGraphicsExtensions{.eps,.pdf,.png,.jpg}
\else
  \DeclareGraphicsExtensions{.eps}
\fi

\newsiamremark{remark}{Remark}
\newsiamremark{hypothesis}{Hypothesis}
\crefname{hypothesis}{Hypothesis}{Hypotheses}
\newsiamthm{claim}{Claim}
\newsiamthm{example}{Example}
\newsiamthm{assumption}{Assumption}

\headers{Completely positive tensor decomposition}{}

\title{An ideal-sparse generalized moment problem reformulation for completely positive tensor decomposition exploiting maximal cliques of multi-hypergraphs \thanks{Submitted to the editors DATE.
\funding{This work was funded by the National Natural Science Foundation of China (12401405).}}}

\author{Pengfei Huang\thanks{School of Mathematics, Hunan University, Changsha, Hunan 410082, China 
		(\email{huangpf@hnu.edu.cn}).}
	\and Minru Bai\thanks{School of Mathematics, Hunan University, Changsha, Hunan 410082, China 
		(\email{minru-bai@163.com}).}
}

\usepackage{amsopn}

\ifpdf
\hypersetup{
  pdftitle={Completely positive tensor decomposition with ideal-sparsity},
  pdfauthor={}
}
\fi

\begin{document}

\maketitle

\begin{abstract}
In this paper, we consider the completely positive tensor decomposition problem with ideal-sparsity. First, we propose an algorithm to generate the maximal cliques of multi-hypergraphs associated with completely positive tensors. This also leads to a necessary condition for tensors to be completely positive. Then, the completely positive tensor decomposition problem is reformulated into an ideal-sparse generalized moment problem. It optimizes over several lower dimensional measure variables supported on the maximal cliques of a multi-hypergraph. The moment-based relaxations are applied to solve the reformulation. The convergence of this ideal-sparse moment hierarchies is studied. Numerical results show that the ideal-sparse problem is faster to compute than the original dense formulation of completely positive tensor decomposition problems. It also illustrates that the new reformulation utilizes sparsity structures that differs from the correlative and term sparsity for completely positive tensor decomposition problems.
\end{abstract}

\begin{keywords}
Completely positive tensor, moment problem, sparsity, multi-hypergraph, semidefinite programming
\end{keywords}

\begin{MSCcodes}
15A69, 44A60, 90C22
\end{MSCcodes}

\section{Introduction}
Let $m$ and $n$ be positive integers. An $m$th order $n$-dimensional tensor $\mathcal{A}$ is an array indexed by integer tuples $(i_1,i_2,\cdots,i_m)$ with $1\le i_1,i_2,\cdots,i_m\le n$, i.e.,
\[\mathcal{A}=(\mathcal{A}_{i_1,i_2,\cdots,i_m})_{1\le i_1,i_2,\cdots,i_m\le n}.\]
Let $T^m(\mathbb{R}^n)$ be the set of all such real tensors. A tensor $\mathcal{A}\in T^m(\mathbb{R}^n)$ is symmetric if each entry $\mathcal{A}_{i_1,i_2,\cdots,i_m}$ is invariant with respect to all permutations of $(i_1,i_2,\cdots,i_m)$. Let $S^m(\mathbb{R}^n)$ be the set of all symmetric tensors. A symmetric tensor $\mathcal{A}\in S^m(\mathbb{R}^n)$ is completely positive \cite{qi2014nonnegative}, if it has a decomposition of the following form: 
\begin{equation}\label{equ:cp}
	\mathcal{A}=\sum_{r=1}^{R}(\mathbf{v}^k)^{\otimes m},
\end{equation}
where  $\mathbf{v}^1,\cdots,\mathbf{v}^R$ are nonnegative vectors. $\mathbf{v}^{\otimes m}$ denotes the $m$th order n-dimensional symmetric outer product tensor such that 
\[(\mathbf{v}^{\otimes m})_{i_1,i_2,\cdots,i_m}=\mathbf{v}_{i_1}\mathbf{v}_{i_2}\cdots\mathbf{v}_{i_m}.\]
We consider the completely positive tensor decomposition problem that for a symmetric tensor $\mathcal{A}$ whether it is completely positive or not.

Completely positive tensor decompositions have applications in computer vision, blind source separation and multiway clustering \cite{shashua2005non,panagakis2021tensor,cichocki2009nonnegative,he2010detecting}. Given the length $R$ of the decomposition, finding a completely positive decomposition can be formulated as an approximation problem with nonnegativity constraints \cite{lim2009nonnegative}. The interested readers can refer to for instance \cite{kolda2015numerical,che2020multiplicative} about the numerical algorithms about this least square problem. For some special classes of tensors, methods are presented to check their completely positivity. For instance, Qi et al. \cite{qi2014nonnegative} proposed a hierarchyical elimination algorithm to obtain a completely positive decomposition of strongly symmetric hierarchically dominated nonnegative tensors. Fan and Zhou \cite{fan2017semidefinite} charaterized the completely positive tensors as a truncated moment sequence. In this way, the problem of checking whether a real symmetric tensor is completely positive is transformed to a truncated moment problem. If it is not completely positive, they can get a certificate for it; if it is, a completely positive decomposition can be obtained. However, a hiarachy of semidefinite relaxations were applied to solve this moment optimization problem. It involves matrices whose size grows very quickly with the relaxation level, dimension and order of tensors, which are usually computationally expensive.
 
To overcome the scalability issue in the hierarchy of semidefinite relaxations, different kinds of sparsity structures are exploited. The first kind is called correlative sparsity, when there are few correlations between the variables of polynomials \cite{waki2006sums}. The second kind is called term sparsity, when there are few monomial terms involved in polynomials \cite{wang2020second}. Both kinds of sparsity can be combined \cite{wang2022cs}. Magron and Wang have developed the software TSSOS that implements correlative and term sparse sum of squares relaxations for the polynomial optimization \cite{magron2021tssos}. We refer to \cite{magron2023sparse} for a nice survey. Recently, Korda et al. \cite{korda2024exploiting} introduced the ideal-sparsity structure, when one optimizes over a measures supported on the variety of an ideal generatedy by quadratic bilinear monomials $x_ix_j$. Based on this, a hierarchies of ideal-sparse moment relaxations is proposed, which is much faster to compute. 

Motivated by the ideal-sparsity generated by quadratic monomials and the zero entry dominance property of completely positive tensor, we aim to construct an ideal-sparse reformulation of the truncated moment problem with ideals generated by $m$th order monomials $x_{i_1}x_{i_2}\cdots x_{i_m}$. To this end, we introduce the maximal clique for multi-hypergraphs associated with completely positive tensors, whose nonedges corresponding to $x_{i_1}x_{i_2}\cdots x_{i_m}=0$. Then, the hierarchies of ideal-sparse moment relaxations can be applied for completely positive tensor decomposition problems.

This paper is organized as follow. Some preliminaries for polynomial optimization and ideal-sparse generalized moment problems are given in \cref{sec:preliminaries}. In \cref{sec:clique}, the maximal cliques of multi-hypergraphs associated with completely positive tensors are studies. The maximal clique generation algorithm is proposed, and a necessary condition for tensors to be completely positive is presented. Section \ref{sec:ideal} studies the ideal-sparse generalized moment problem reformulation and resulted sparse moment relaxations for the completely positive tensor decomposition problem. The aymptotic and finite convergence are also discussed in this section. Numerical experiments are presented in \cref{sec:numerical}, and we conclude our approach and discuss future work in \cref{sec:conclusion}.

\section{Preliminaries}\label{sec:preliminaries}
\subsection{Notation}
We denote scalars, vectors, and tensors by lowercase letters, bold lowercase letters, and calligraphic letters, respectively. For a tensor $\mathcal{A}\in T^m(\mathbb{R}^n)$, we also use $\mathcal{A}_{i_1,\cdots,i_m}$ to denote its entry. For a symmetric matrix $X$, $X\succeq 0$ means $X$ is positive semidefinite. $tr(X)$ means the trace of $X$. $\|\cdot\|$ denotes the 2-norm. For $t\in\mathbb{R}$, $\lceil t\rceil$ dentoes the smallest integer not smaller than $t$. Let $\mathbf{x}=(x_1,\cdots,x_n)$ be a tuple of varialbes. $\mathbb{R}[\mathbf{x}]$ (resp., $\mathbb{R}[\mathbf{x}]_d$) denotes the set of polynomials in $\mathbf{x}$ with real coefficients (with degree most $d$). Set $\mathbb{N}^n_d=\{\bm{\alpha}\in\mathbb{N}^n:|\bf{\alpha}|\le d\}$, where $|\bm{\alpha}|=\sum_{i=1}^n\alpha_i$. Denote by $\text{deg}(p)$ the degree of a polynomial $p$. For  $\bm{\alpha}=(\alpha_1,\cdots,\alpha_n)\in\mathbb{N}^n$, denote the monomial power $\mathbf{x}^{\bm{\alpha}}:=x_1^{\alpha_1}\cdots x_n^{\alpha_n}$. We use $[\mathbf{x}]_d$ to denote the vector of all monomials in $\mathbf{x}$ whose degree is at most $d$ (listed in some given order). Let $\mathcal{A}\in S^m(\mathbb{R}^n)$, denote by $\mathbb{R}[\mathbf{x}]_{\mathcal{A}}$ the set of polynomials with monomials whose exponents are from $\{\bm{\alpha}\in\mathbb{N}^m_n:\bm{\alpha}=\mathbf{e}_{i_1}+\mathbf{e}_{i_2}+\cdots+\mathbf{e}_{i_m}, 1\le i_1\le\cdots\le i_m\le n\}$ and $\mathbf{e}_i$ is the $i$-th unit vector in $\mathbb{R}^n$. 

A polynomial $p\in\mathbb{R}[\mathbf{x}]$ is said to be a sum of squares (SOS) if there exists $\sigma_1,\cdots,\sigma_t\in\mathbb{R}[\mathbf{x}]$ such that $p=(\sigma_1)^2+\cdots+(\sigma_t)^2$. Denote by $\Sigma[\mathbf{x}]$ the set of SOS polynomials in $\mathbf{x}$, and let $\Sigma[\mathbf{x}]_d:=\Sigma[\mathbf{x}]\cap\mathbb{R}[\mathbf{x}]_d$. Given a tuple $\bm{g}=(g_1,\cdots,g_m)\subseteq\mathbb{R}[\mathbf{x}]$, the quadratic module of $\mathbb{R}[\mathbf{x}]$ generated by $\bm{g}$ is the set
\[Q(\bm{g}):=\Sigma[\mathbf{x}]+g_1\cdot\Sigma[\mathbf{x}]+\cdots+g_m\cdot\Sigma[\mathbf{x}],\]
and the $2d$th truncation of $Q(\bm{g})$ is the set
\begin{equation*}\label{equ:qmod}
	Q(\bm{g})_{2d}:=\Sigma[\mathbf{x}]_{2d}+g_1\cdot\Sigma[\mathbf{x}]_{2d-\text{deg}(g_1)}+\cdots+g_m\cdot\Sigma[\mathbf{x}]_{2d-\text{deg}(g_m)}.
\end{equation*}
For a tuple $\bm{h}=(h_1,\cdots,h_l)\subseteq\mathbb{R}[\mathbf{x}]$, the ideal of $\mathbb{R}[\mathbf{x}]$ generated by $\bm{h}$ is the set
\[I(\bm{h}):=h_1\cdot\mathbb{R}[\mathbf{x}]+\cdots+h_l\cdot\mathbb{R}[\mathbf{x}],\]
and the $2d$th truncation of $I(\bm{h})$ is the set
\begin{equation*}\label{equ:ideal}
	I(\bm{h})_{2d}:=h_1\cdot\mathbb{R}[\mathbf{x}]_{2d-\text{deg}(h_1)}+\cdots+h_l\cdot\mathbb{R}[\mathbf{x}]_{2d-\text{deg}(h_1)}. 
\end{equation*}

Denote $[n]:=\{1,\cdots,n\}$ and let $[\{j_1,\cdots,j_m\}]$ denote the set of all distinct elements in $\{j_1,\cdots,j_m\}$. Consider a set $U\subset[n]$, $|U|$ means the number of elements in $U$, let $\mathbf{x}(U)=\{x_i:i\in U\}$. For a vector $\mathbf{y}\in\mathbb{R}^{|U|}$, $\hat{\mathbf{y}}\in\mathbb{R}^n$ is defined as $\hat{y}_i=y_i$ if $i\in U$ and $\hat{y}_i=0$ otherwise. The support of a vector $\mathbf{x}\in\mathbb{R}^n$ is the set Supp$(\mathbf{x})=\{i\in[n]: x_i\neq 0\}$. For a function $f:~\mathbb{R}^n\rightarrow\mathbb{R}$, we let $f|_{U}:\mathbb{R}^{|U|}\rightarrow\mathbb{R}$ denote the function in the variables $\mathbf{x}(U)$, defined as $f|_{U}(\mathbf{x}(U))=f(\widehat{\mathbf{x}(U)})$.

\subsection{The $\mathcal{A}$-truncated moment problem}
For a tensor $\mathcal{A}\in S^m(\mathbb{R}^n)$, let $d>m/2$, the completely positive tensor decomposition problem can be transformed into the following linear optimization problem \cite{fan2017semidefinite,nie2014truncated}
\begin{equation}\label{equ:dense GMP}
	\begin{aligned}
		val:=\underset{\mu\in\mathcal{M}(\mathbb{R}^n)}{\min}&\int Fd{\mu}\\
		\mbox{s.t.}&\int \mathbf{x}^{\bm{\alpha}}d{\mu}=\mathcal{A}_{i_1,i_2,\cdots,i_m}~(1\le i_1\le\cdots\le i_m\le n)\\
		&\text{Supp}(\mu)\subset K,
	\end{aligned}
\end{equation}
where $K=\{\mathbf{x}\in\mathbb{R}^n: \mathbf{x}^T\mathbf{x}-1=0,x_1\ge 0,x_2\ge0,\cdots,x_n\ge0\}.$ $\mathcal{M}(\mathbb{R}^n)$ is the set of nonnegative Borel measure on $\mathbb{R}^n$, $F\in \mathbb{R}[\mathbf{x}]_d$ is a random SOS polynomial in variables $\mathbf{x}=(x_1,\cdots,x_n)$ and $\bm{\alpha}=\mathbf{e}_{i_1}+\mathbf{e}_{i_2}+\cdots+\mathbf{e}_{i_m}$. 

Let $\mathbb{R}^{\mathbb{N}^n_d}$ be the set of real vectors indexed by $\bm{\alpha}\in\mathbb{N}^n_d$. A vector in $\mathbb{R}^{\mathbb{N}^n_d}$ is called a truncated multisequence (tms) of degree $d$. For $F\in\mathbb{R}[\mathbf{x}]_d$ and $\mathbf{z}\in\mathbb{R}^{\mathbb{N}^n_d}$, define the scalar product
\[\langle\sum_{\bm{\alpha}\in\mathbb{N}^n_d}F_{\bm{\alpha}}\mathbf{x}^{\bm{\alpha}},\mathbf{z}\rangle:=\sum_{\alpha\in\mathbb{N}^n_d}F_{\bm{\alpha}}{z}_{\bm{\alpha}},\]
where each $F_{\bm{\alpha}}$ is a coefficient of $F$. Then, \cref{equ:dense GMP} is equivalent to a $\mathcal{A}$-truncated moment problem \cite{nie2014truncated}
\begin{equation}\label{equ:truncated}
	\begin{aligned}
		val:=\underset{\mathbf{z}\in\mathbb{R}^{\mathbb{N}^n_d}}{\min}&\langle F,\mathbf{z}\rangle \\
		\mbox{s.t.}& z_{\bm{\alpha}}=\mathcal{A}_{i_1,i_2,\cdots,i_m}~(1\le i_1\le\cdots\le i_m\le n)\\
		&\mathbf{z}\in \mathcal{R}_d,
	\end{aligned}
\end{equation}
where 
\[\mathcal{R}_d:=\left\{\mathbf{z}\in\mathbb{R}^{\mathbb{N}^n_d}:\exists \mu\in\mathcal{M}(\mathbb{R}^n), \text{Supp}(\mu)\subset K \text{ such that }z_{\bm{\alpha}}=\int\mathbf{x}^{\bm{\alpha}}d\mu \text{ for } \bm{\alpha}\in\mathbb{N}^{n}_d\right\}.\]
We say that $\mathbf{z}$ admits a representing measure supported in a set $K$ if $\mathbf{z}\in\mathcal{R}_d$. A mesture is finitely atomic if its support is a finite set and is $r$-atomic if its support consists of at most $r$ distinct points. Hence, a tensor $\mathcal{A}$ is completely positive if and only if there exists $\mathbf{z}\in\mathcal{R}_d$ admits a finitely atomic measure. In fact, there exist $\lambda_1,\cdots,\lambda_R>0$ and nonnegative vectors $\mathbf{u}^1,\cdots,\mathbf{u}^R$ such that 
\[\mathcal{A}=\lambda_1(\mathbf{u}^1)^{\otimes m}+\cdots+(\mathbf{u}^R)^{\otimes m}\]
 if and only if $\mathbf{z}\in\mathcal{R}_d$ admits a measure
  \[\mu=\lambda_i\delta_{\mathbf{u}^i},\]
 where $\delta_{\mathbf{u}}$ is the dirac measure.

\subsection{Dense moment relaxations}

Since the optimization problem \cref{equ:truncated} is difficult to deal with directly, a common strategy is to build hierarchies of moment approximations for it. 

For a polynomial $q\in\mathbb{R}[\mathbf{x}]_{2t}$, the $t$th localizing matrix of $q$, generated by a tms $\mathbf{z}\in\mathbb{R}^{\mathbb{N}^n_{2t}}$, is a symmetric matrix $L_q^{(t)}(\mathbf{z})$ such that
$\langle qp^2,\mathbf{z}\rangle=vec(p)^T(L_q^{(t)}(\mathbf{z}))vec(p)$ for all $p\in\mathbb{R}[\mathbf{x}]_{t-\lceil\text{deg}(q)/2\rceil}$, where $vec(p)$ is the coefficient vector of $p$. That is, \[L_q^{(t)}(\mathbf{z})=\left(\sum_{\bm{\alpha}}q_{\bm{\alpha}}y_{\bm{\alpha}+\bm{\beta}+\bm{\gamma}}\right)_{\bm{\beta},\bm{\gamma}\in\mathbb{N}^n_{t-\lceil\text{deg}(q)/2\rceil}}.\]
And the moment matrix is denoted as
\[M_t(\mathbf{z}):=L_1^{(t)}(\mathbf{z}).\]

The moment relaxation of level $t$ for problem \eqref{equ:truncated} is
\begin{equation}\label{equ:dense relax}
	\begin{aligned}
	\xi_t:=\underset{{\mathbf{z}}\in\mathbb{R}^{\mathbb{N}^n_{2t}}}{\min}&\langle F,\mathbf{z}\rangle\\
	\mbox{s.t.}&{z}_{\bm{\alpha}}=\mathcal{A}_{i_1,i_2,\cdots,i_m}~(1\le i_1\le\cdots\le i_m\le n)\\
	&M_t(\mathbf{z})\succeq 0\\
	&L_{x_j}^{(t)}(\mathbf{z})\succeq 0~(j\in[n])\\
	&L_{\mathbf{x}^T\mathbf{x}-1}^{(t)}(\mathbf{z})=0.
	\end{aligned}
\end{equation}
The dual optimization problem of \cref{equ:dense relax} is 
\begin{equation}\label{equ:dense dual}
	\begin{aligned}
		\underset{P\in\mathbb{R}[\mathbf{x}]_{\mathcal{A}}}{\max}&\langle P,\mathbf{w}\rangle\\
		\text{s.t.}&F-P\in I((\mathbf{x}^T\mathbf{x}-1))_{2t}+Q((x_1,\cdots,x_n))_{t},
	\end{aligned}
\end{equation}
where $\mathbf{w}$ is given as $w_{\bm{\alpha}}=\mathcal{A}_{i_1,\cdots,i_m}$ for any $\bm{\alpha}=\mathbf{e}_{i_1}+\mathbf{e}_{i_2}+\cdots+\mathbf{e}_{i_m}$.
We refer to the above hierarchy of parameter $\xi_t$ as dense moment hierarchy. It is easy to see that $\xi_t\le\xi_{t+1}\le\xi_{\infty}\le$val. We refer the reader to \cite{nie2023moment} for more references about moment and polynomial optimizations.

Due to the work of Curto and Fialkow \cite{curto1996solution,curto2000truncated}, a finitely atomic optimal solution $\mu$ of \cref{equ:dense GMP} can be extracted from tms $\mathbf{z}$ the solution of \cref{equ:dense relax} under the flatness condition.
\begin{theorem}[\cite{curto1996solution,curto2000truncated}]\label{thm:flat}
	Let $t\in\mathbb{N}$ such that $t\ge t_0=\max\{\lceil \text{deg}(F)/2\rceil,\lceil m/2\rceil,1\}$ and set $d_K=\max\{1,\lceil \text{deg}(x_j)/2\rceil~(j\in[n]),\lceil\text{deg}(\mathbf{x}^T\mathbf{x}-1)/2\rceil\}$. Assume $\mathbf{z}\in\mathbb{R}^{\mathbb{N}^n_{2t}}$ is an optimal solution to the moment relaxation of lever $t$ \cref{equ:dense relax} and it satisfies the following flatness condition:
	\begin{equation}\label{equ:rank condition}
		rank~M_s(\mathbf{z})=rank~M_{s-d_K}(\mathbf{z}):=r
	\end{equation} 
for some integer $s$ such that $t_0\le s\le t$. Then equality $\xi_t=$val holds, and \eqref{equ:dense GMP} has an optimal solution $\mu$ that is $r$-atomic. 
\end{theorem}

For a tms $\mathbf{z}\in\mathbb{R}^{\mathbb{N}^n_t}$, denote by $\mathbf{z}|_s$ the subvector of $\mathbf{z}$ whose indices belong to $\mathbb{N}^n_s$. Now let us recall the theorem that guarantees the achievement of checking whether a symmetric tensor $\mathcal{A}$ is completely positive and extracting a decomposition if it is.
\begin{theorem}[\cite{nie2014truncated,fan2017semidefinite}]\label{thm:nie}
		Let $\mathcal{A}\in S^m(\mathbb{R}^n)$. The moment hierarchies \cref{equ:dense relax} has the following properties:
	\begin{enumerate}
		\item If \cref{equ:dense relax} is infeasible for some $t$, then $\mathcal{A}$ is not completely positive.
		\item If $\mathcal{A}$ is not completely positive, then \cref{equ:dense relax} is infeasible for all $t$ big enough.
		\item If $\mathcal{A}$ is completely positive, then for almost all generated $F$, \cref{equ:dense relax} has an optimizer $\mathbf{z}^{*,t}$ for all $t>\text{m}/2$. 
		\item For all $s$ big enough, the sequence $\{\mathbf{z}^{*,t}|_{2s}\}_t$ is bounded and all its accumulation points are flat.
		\item Under mild conditions, \cref{equ:dense relax} has finite convergence, that is $\mathbf{z}^{*,t}|_{2s}$ is flat for some $s> m/2$ for $t$ big enough.
	\end{enumerate}
\end{theorem}
\begin{remark}
	The mild conditions for finite convergence can be refered to \cite{nie2014truncated,nie2014optimality}.
\end{remark}

However, a common drawback of the dense moment hierarchy \cref{equ:dense relax} is that sizes of the matrices grow quickly with the level $t$ and with the order and dimension of the tensor $\mathcal{A}$.

\subsection{Ideal-sparsity for generalized moment problem}
Korda et al. \cite{korda2024exploiting} has considered the generalized moment problem of the form
\begin{equation}\label{equ:quadratic_GMP}
	\begin{aligned}
	val:=\inf_{\mu\in\mathcal{M}(\mathbb{R}^n)}&\int f_0d\mu\\
	\text{s.t.}&\int f_id\mu=a_i(i\in[N])\\
	&\text{Supp}(\mu)\subseteq K,
	\end{aligned}
\end{equation}
where
\begin{equation}\label{equ:quadratic}
	K=\{\mathbf{x}\in\mathbb{R}^n:g_j(\mathbf{x})\ge0~(j\in[m]),~x_ix_j=0~(\{i,j\}\in\bar{\mathbb{E}})\}. 
\end{equation}
$\mathbb{E}$ is the edge set of a graph $G=([n], \mathbb{E})$ and $\bar{\mathbb{E}}$ corresponds to the set of nonedges. We call $G$ the support graph of \cref{equ:quadratic_GMP}. A set $U\subseteq [n]$ is a clique of the graph $G$ if $\{i,j\}\in \mathbb{E}$ for any two distinct vertices $i,j\in U$. If a clique is not strictly contained in any other clique of $G$, it is maximal.

Then the ideal constraints generated by $x_ix_j=0$ in \cref{equ:quadratic}, imply that the support points of $\mu$ should be sparse. Korda et al. \cite{korda2024exploiting} refer to this sparsity structure on \eqref{equ:quadratic_GMP} as ideal-sparsity. They construct the following equivalent sparse problem of \cref{equ:quadratic_GMP}:
\begin{equation}\label{equ:quadratic_sparse}
	\begin{aligned}
		val^{isp}:=\inf_{\mu_k\in\mathcal{M}(\mathbb{R}^{|V_k|}),k\in[p]}&\sum_{k=1}^p\int f_0|_{V_k}d\mu_k\\
		\text{s.t.}&\sum_{k=1}^p\int f_i|_{V_k}d\mu_k=a_i(i\in[N])\\
		&\text{Supp}(\mu_k)\subseteq K_k~(k\in[p]),
	\end{aligned}
\end{equation}
where $V_k~(k\in [p])$ is the maximal clique of $G$ and 
\[K_k:=\{\mathbf{y}\in\mathbb{R}^{|V_k|}:\hat{\mathbf{y}}\in K,~\text{Supp}(\hat{\mathbf{y}})\subseteq V_k\}.\]

Therefore, the sparse generalized moment problem \cref{equ:quadratic_sparse} may optimize over several measures that are supported on dimensional spaces smaller than $\mathbb{R}^n$. As a result, the ideal-sparse moment hierarchy can be established, in which the involved positive semidefinite matrices are of much smaller size, compared with the dense moment relaxations.

\section{The uniform multi-hypergraph and its clique}
We consider the completely positive tensor $\mathcal{A}\in S^m(\mathbb{R}^n)$ with $\mathcal{A}_{i_1,\cdots,i_m}=0$ for some $1\le i_1\le\cdots \le i_m\le n$.  In such case, as will be demonstrated in \cref{sec:ideal}, problem \cref{equ:dense GMP} is equivalent to a generalized moment problem restricted on \[K\subseteq\{\mathbf{x}\in\mathbb{R}^n:x_{i_1}\cdots x_{i_m}=0 \text{ for }\mathcal{A}_{i_1,\cdots,i_m}=0\}.\] It is natural to expect that the ideal-sparsity can be also applied when the ideal is genereated by $x_{i_1}\cdots x_{i_m}$. And then we can deal with similar ideal-sparse moment hierarchies instead of dense moment relaxations \cref{equ:dense relax} to solve the completely positive tensor decomposition problem. However, this extension is not as straightforward as it appears.

\subsection{Difficulties from $x_ix_j$ to $x_{i_1}\cdots x_{i_m}$}
The construction of the sparse problem \cref{equ:quadratic_sparse} relies on maximal cliques of the associated support graph. As pointed by Korda el al. \cite{korda2024exploiting}, $V_1,\cdots,V_p$ indeed denote the maximal subsets of $[n]$ that do not contain any set $\{i_1,\cdots,i_m\}$ with $x_{i_1}\cdots x_{i_m}=0$. 

First, it is obvious that $x_{i_1}\cdots x_{i_m}=0(m\ge 3)$ is not equal to $x_{j_1}x_{j_2}=0$ for all distinct $j_1,j_2\in \{i_1,\cdots,i_m\}$ when $[\{i_1,\cdots,i_m\}]$ has more than two elements. This often occurs in competely positive tensors, which can be seen in the following example.
\begin{example}
	Consider the symmetric tensor $\mathcal{A}\in S^3(\mathbb{R}^3)$. It has the completely positive decomposition
	\[\mathcal{A}=\begin{bmatrix}1\\1\\0\end{bmatrix}^{\otimes 3}+\begin{bmatrix}0\\1\\1\end{bmatrix}^{\otimes 3}+\begin{bmatrix}1\\0\\1\end{bmatrix}^{\otimes 3}.\]
It leads to the constraint $x_1x_2x_3=0$ while $x_1x_2,x_2x_3,x_1x_3\neq 0$.
\end{example}

Therefore, for $m\ge 3$ one may need another approach to characterize the relationship $x_{i_1}\cdots x_{i_m}=0$. It has been known that a hypergraph $G$ contains edges that connect more than two nodes. However, if we want to define a similar support graph as the one for \cref{equ:quadratic_GMP}, $x_{i_1}\cdots x_{i_m}\neq 0$ indicates an edge connecting $\{i_1,\cdots,i_m\}$, which may contains repeat vertices (often called a hyperloop). Then conflict may arise in the definition of cliques, when for example $\mathcal{A}_{112}=0$ while $\mathcal{A}_{122}\neq 0$.

Another concern related to the approach expoliting ideal-sparsity is that how to obtain ``maximal cliques" for a hypergraph, since the maximum clique problem for graph has been known to be an NP-complete problem \cite{wu2015review}.

In the next two subsections, we introduce the definition of the maximal clique for a multi-hypergraph associated with completely positive tensors. And an algorithm to generate all such maximal cliques is proposed, which will be used to exploit the ideal-sparsity.
\subsection{The uniform multi-hypergraph}
We first recall and introduce the uniform multi-hypergraph and its clique.
\begin{definition}[\cite{pearson2014spectral}]
	Let $V=\{v_1,v_2,\cdots,v_n\}$. A multi-hypergraph $G$ is a pair $(V,\mathbb{E})$, where $\mathbb{E}=\{E_1,\cdots,E_N\}$ is a set of multisets (a set in which members are allowed to appear more than once) of $V$. The elements of $V$ are called the vertices and the elements of $\mathbb{E}$ are called the edges. Moreover, a multi-hypergraph is called $m$-uniform $(m\ge 2)$ if for all $E\in\mathbb{E}$, $|E|=m$ (including repeated memberships).
\end{definition}
\begin{definition}[\cite{xu20160}]
	Let $V=[n]$. A tensor $\mathcal{A}\in S^m(\mathbb{R}^n)$ is said to be an associated tensor with the $m$-uniform multi-hypergraph $G=(V,\mathbb{E})$ if $\mathcal{A}_{i_1,\cdots,i_m}\neq 0$ for all $\{i_1,\cdots,i_m\}\in\mathbb{E}$, and $\mathcal{A}_{i_1,\cdots,i_m}=0$ otherwise. We call $G$ the support multi-hypergraph of $\mathcal{A}$.
\end{definition}

\begin{lemma}[Zero entry dominance property \cite{qi2017tensor}]\label{lem:dominance}
	Suppose that $\mathcal{A}\in S^m(\mathbb{R}^n)$ is completely positive, for any two indices $\{i_1,\cdots,i_m\}$ and $\{j_1,\cdots,j_m\}$, $i_1\le\cdots\le i_m$, $j_1\le\cdots\le j_m$, if $[\{i_1,\cdots,i_m\}]\subseteq [\{j_1,\cdots,j_m\}]$,  and $\mathcal{A}_{j_1,\cdots,j_m}>0$, then $\mathcal{A}_{i_1,\cdots,i_m}>0$.
\end{lemma}

Motivated by the $m$-uniform clique of the hypergraph \cite{torres2017hclique} and the zero dominance property of completely positive tensors, we define the clique for a support multi-hypergraph of completely positive tensors as following.
\begin{definition}[Clique and maximal clique]\label{def:clique}
	Let $G=(V,\mathbb{E})$ be an $m$-uniform multi-hypergraph corresponding to completely positive tensors. The clique of $G$ is a subset (without repeat members) $J\subset V$, such that any multiset $E\subset J$ with $|E|=m$ (including repeated memberships) is an edge of $G$. And a clique is maximal if it is not strictly contained in any other clique of $G$.
\end{definition}
The clique in the above definition is well defined in the sense that for any indice $\{i_1,\cdots,i_m\}$, it forms an edge of the support multi-hypergraph $G$ of completely positive tensors if and only if $[\{i_1,\cdots,i_m\}]$ belongs to some clique of $G$. This fact is evident due to \cref{lem:dominance}.
\begin{remark}
	 Xu el al. \cite{xu20160} have defined the clustering of edges for a support multi-hypergraph $G$, denoted by $\mathcal{D}_1,\cdots,\mathcal{D}_r$. We point out that they correspond one-to-one with maximal cliques $V_1,\cdots,V_r$ of $G$. In fact, $V_i=[l_i]$ where $l_i$ is the maiximal element of $\mathcal{D}_i$.
\end{remark}

\subsection{Maximal clique algorithm}\label{sec:clique}
One major issue to be solved exploring the ideal sparsity in tensor completely decomposition problem is how to determine maximal cliques of the associated support multi-hypergraph. The difficulty different from the graph maximal clique problem, comes from that in a multi-hypergraph, the vertices in an edge $E$ can be repeated and $|E|>2$ in general. 

Nevertheless, benefit from the zero entry dominance property, we can propose a preliminary algorithm to generate maximal cliques of a uniform multi-hypergraph associated with the completely positive tensor, utilizing the clique introduced in \cref{def:clique}. The algorithm finding all maximal cliques is presented below in \cref{alg:clique}.
\begin{algorithm}[htb]
	\caption{Maximal cliques generation}\label{alg:clique}
	\begin{algorithmic}[1]
		\State Given a completely positive tensor $\mathcal{A}\in S^m(\mathbb{R}^n)$, and an initial set of candidates $C=\{[n]\}$.
		\For{$1\le i_1\le\cdots\le i_m\le n$}
		\State (Split Procedure, step 4.-step 20.)
		\If{$\mathcal{A}_{i_1,\cdots,i_m}=0$}
		\For{$S\in C$}
		\If{$[\{i_1,\cdots,i_m\}]\subseteq S$}
		\State $C:=C\backslash \{S\}$ \Comment{exclude impossible candidates}
		\State $add\_C:=\emptyset$
		\For{$k\in [\{i_1,\cdots,i_m\}]$}
		\State $S\_new:=S\backslash\{k\}$
		\State $add\_C:=add\_C\cup\{S\_new\}$
		\EndFor
		\EndIf
		\EndFor
		\For{$S\_new\in add\_C$}
		\If{$S\_new\not\subseteq S$ for any $S\in C$}\Comment{add possible candidates}
		\State $C:=C\cup\{S\_new\}$
		\EndIf
		\EndFor
		\EndIf 
		\EndFor
		\State \textbf{Return} $C\_end := C$ the set of all maximal cliques.
	\end{algorithmic}
\end{algorithm}

\begin{proposition}\label{pro:clique}
	Suppose $\mathcal{A}\in S^m(\mathbb{R}^n)$ is complete positive, $G=([n],\mathbb{E})$ is the support multi-hypergraph of $\mathcal{A}$. Then \cref{alg:clique} can generate all maximal cliques of $G$.
\end{proposition}
\begin{proof}
	Let $V_i$, $i\in[p]$ denote all the maximal cliques of $G$. We will prove this proposition by induction. With the initial set given as $C=\{[n]\}$, it is obvious that $V_i\subseteq [n], ~i\in[p]$. If $\mathcal{A}_{i_1,\cdots,i_m}\neq 0$ for any $1\le i_1\le\cdots\i_m\le n$, then $[n]$ is a maximal clique and the Split Procedure will not be conducted. \cref{alg:clique} finds all maximal cliques trivially.
	
	Otherwise, there exists $\mathcal{A}_{i_1,\cdots,i_m}=0$ for some $1\le i_1\le\cdots\le i_m\le n$. Suppose that $[\{i_1,\cdots,i_m\}]=\{j_1,\cdots,j_l\}$ and $j_k\neq j_t$ when $k\neq t$. Consider executing the Split Procedure once, then $C=\{[n]\}$ is updated to $C=\{S_1,\cdots,S_l\}$ where $S_k=[n]\backslash\{j_k\}$. Clearly,  $S_k\not\subseteq S_t$ when $k\neq t$. For each maximal clique $V_i$, we also have that $V_i\subseteq S_k$ for some $k$. In fact, if $V_i\not\subseteq S_k$ for any $k$, then $j_k\in V_i$ for $k\in[l]$. It leads that $\{i_1,\cdots,i_m\}\in\mathbb{E}$, which conflicts. And it is obvious that $[\{i_1,\cdots,i_m\}]$ is not included in any $S_k$ for $k\in[l]$.
	
	Assume that after executing the Split Procedure $k$ times, the current candidate set $C=\{S_1,\cdots,S_l\}$. It satisfies $S_k\not\subseteq S_t$ when $k\neq t$, and each $V_i\subseteq S_k$ for some $k$. If $S_t$ is a clique, then it is maximal, otherwise $S_t\subseteq V_i\subseteq S_k$ for some $i$ and $k$, which is a contradiction. If every $S_t$ is maximal, then \cref{alg:clique} finds $C\_end=C$, which is exactly the set of all maximal cliques. If there is a $S_k$ that is not a clique, then there exists $[{i_1,\cdots,i_m}]\subseteq S_k$ such that $\mathcal{A}_{i_1,\cdots,i_m}=0~(1\le i_1\le\cdots\le i_m\le n)$ and has not been checked by \cref{alg:clique}. Thus, the Split Procedure will be conducted again. Suppose that $C$ is updated to $C=\{\hat{S}_1,\cdots,\hat{S}_{\hat{l}}\}$. Now, we prove that $\hat{S}_k\not\subseteq \hat{S}_t$ when $k\neq t$, and each $V_i\subseteq \hat{S}_k$ for some $k$.
	
	From step 6 and step 18, $\hat{S}_k$ is either $S_t$ for some $t\in[l]$ or a new generated set $S\_{new}$. And $S\_{new}$ will only be added to $C$ if it is not included in any of existing items of $C$ from step 16-step 17. Hence, $\hat{S}_k$ are still not mutually contained. Suppose that $V_i\subseteq S_k$. If $[\{i_1,\cdots,i_m\}]\not\subseteq S_k$, then $S_k$ is conserved in the updated $C$. If $[\{i_1,\cdots,i_m\}]=\{j_1,\cdots,j_l\}\subseteq S_k$, then we have $S_k\backslash\{j_s\}\in add\_C$ for $s\in [l]$ from step 9-step 12. We claim that $V_i\subseteq S_k\backslash\{j_s\}$ for some $s$. Otherwise, $\{j_1,\cdots,j_l\}\subseteq V_i$ and it conflicts. The desired result follows by induction.
\end{proof}

\begin{corollary}\label{thm:maximal sets}
	For any symmetric tensor $\mathcal{A}$, whether it is completely positive or not, $V_1,\cdots,V_p$ generated by \cref{alg:clique} indeed are the maximal subsets of $[n]$ that do not contain any set $\{i_1,\cdots,i_m\}$ with $\mathcal{A}_{i_1,\cdots,i_m}=0$.
\end{corollary}
This result can be obtained through the proof of \cref{pro:clique} trivially.

For a completely positive decomposable tensor $\mathcal{A}\in S^m(\mathbb{R}^n)$, \cref{alg:clique} can find all its maximal cliques from \cref{pro:clique}, so that $[\{i_1,\cdots,i_m\}]$ should be included in at least one maximal cliques for any $\mathcal{A}_{i_1,\cdots,i_m}\neq 0$. Conversely, for a nonnegative symmetric tensor, we can still implete \cref{alg:clique} for it first. If there exists $[\{i_1,\cdots,i_m\}]$ not included in any generated ``maximal clique", for some $\mathcal{A}_{i_1,\cdots,i_m}\neq 0$, this tensor is not completely positive decomposable. Hence, it implies a necessary condition for a nonnegative symmetric tensor to be competely positive as presented below.
\begin{corollary}\label{thm:necessary}
	Let $\mathcal{A}\in S^m(\mathbb{R}^n)$, $V_1,\cdots,V_p$ are generated by \cref{alg:clique} applied to $\mathcal{A}$. If $\mathcal{A}$ is completely positive, then for any $\mathcal{A}_{i_1,\cdots,i_m}\neq 0$, $[\{i_1,\cdots,i_m\}]\in V_k$ for some $k$. Otherwise, $\mathcal{A}$ is not completely positive.
\end{corollary}

\begin{remark}
	Although the idea is to check all possible subsets of $[n]$ and find maximal cliques, \cref{alg:clique} might be more promising than expected. Due to the sparsity and zero entry dominance property of complete positive tensors, the Split Procedure can be conducted efficiently. For instance, let $\mathcal{A}\in S^3(\mathbb{R}^n)$ be complete positive. If $\mathcal{A}_{112}=0$, the Split Procedure will update $C$ such that $\{1,2\}\not\subseteq S_k$ for any $S_k\in C$. Since \cref{alg:clique} check $\{i_1,\cdots,i_m\}$ in order, for all subsequent $\mathcal{A}_{i_1,\cdots,i_m}=0$ with $\{1,2\}\subseteq\{i_1,\cdots,i_m\}$, like $\mathcal{A}_{123}$, there is no $S_k\in C$ than include $[\{i_1,\cdots,i_m\}]$. As a result, the Split Procedure may require much less operations. Please see more experiment examples in \cref{subsec:clique}.
\end{remark}

Here we provide a simple example to illustrate how \cref{alg:clique} works.
\begin{example}[\cite{fan2017semidefinite}]\label{exa:example}
	Consider the tensor $\mathcal{A}\in S^3(\mathbb{R}^3)$ given as:
	\[\mathcal{A}_{1,1,1}=2,\mathcal{A}_{1,1,2}=1,\mathcal{A}_{1,1,3}=1,\mathcal{A}_{1,2,2}=1.\mathcal{A}_{1,2,3}=0,\]
	\[\mathcal{A}_{1,3,3}=1,\mathcal{A}_{2,2,2}=2,\mathcal{A}_{2,2,3}=0,\mathcal{A}_{2,3,3}=0,\mathcal{A}_{3,3,3}=1.\]
	For convenience, we only list the value with indices $1\le i_1\le \cdots\le i_m\le n$ since $\mathcal{A}$ is symmetric. It has been proved to be completely positive. \cref{tab:example} presents the algorithm process for this example. `zero indice' is the indices of zero entries. `C' is the initial state of the candidate set when checking each zero entry. The elements with an underline are sets that will be excluded in the Split Procedure. The generated maximal cliques are $\{\{1,3\},\{1,2\}\}$.
	\begin{table}[H]\label{tab:example}
		\caption{Illustration of \cref{alg:clique}}
		\begin{center}
		\begin{tabular}{lll}
			\hline
			zero indice & $C$                                       & $add\_C$                      \\ \hline
			$\{1,2,3\}$ & $\{\underline{\{1,2,3\}}\}$                   & $\{\{2,3\},\{1,3\},\{1,2\}\}$ \\
			$\{2,2,3\}$ & $\{\underline{\{2,3\}},\{1,3\},\{1,2\}\}$ & $\{\{2\},\{3\}\}$             \\
			$\{2,3,3\}$ & $\{\{1,3\},\{1,2\}\}$                     & $\emptyset$                      \\
		 \hline
		\end{tabular}
		\end{center}
	\end{table}
\end{example}

\section{Ideal-sparse reformulation and sparse moment relaxations}\label{sec:ideal}
Throughout this section, $\bm{\alpha}=\mathbf{e}_{i_1}+\mathbf{e}_{i_2}+\cdots+\mathbf{e}_{i_m}$, unless otherwise specified. Let $G=(V,\mathbb{E})$ be the support multi-hypergraph of $\mathcal{A}\in S^m(\mathbb{R}^n)$. One can exploit the fact that \cref{equ:dense GMP} is equivalent to 
\begin{equation}\label{equ:dense ideal}
	\begin{aligned}
		\underset{\mu\in\mathcal{M}(\mathbb{R}^n)}{\min}&\int Fd{\mu}\\
		\mbox{s.t.}&\int \textbf{x}^{\bm{\alpha}}d{\mu}=\mathcal{A}_{i_1,i_2,\cdots,i_m}~(1\le i_1\le\cdots\le i_m\le n)\\
		&\text{Supp}(\mu)\subseteq K,
	\end{aligned}
\end{equation}
where 
\begin{equation}\label{equ:ideal K}
K=\{\mathbf{x}\in\mathbb{R}^n: \mathbf{x}^T\mathbf{x}-1=0,x_j\ge 0~j\in[n], x_{i_1}\cdots x_{i_m}=0~\{i_1,\cdots,i_m\}\in \bar{\mathbb{E}}\}.
\end{equation}
\begin{lemma}\label{lem:equal}
	Problems \cref{equ:dense GMP} and \cref{equ:dense ideal} are equivalent.
\end{lemma}
\begin{proof}
	First, assume $\mu$ is feasible for problem \cref{equ:dense ideal}. It is obvious that $\mu$ is also feasible for problem \cref{equ:dense GMP}.
	
	We now show the reverse. Assume $\mu$ is feasible for \cref{equ:dense GMP}. If it is not feasible for \cref{equ:dense ideal}, then there exists $\bar{\mathbf{x}}\in \text{Supp}(\mu)$, such that $\bar{x}_{i_1}\cdots \bar{x}_{i_m}>0$ for some $\{i_1,\cdots,i_m\}\in\bar{\mathbb{E}}$ and $\int x_{i_1}\cdots x_{i_m}d\mu>0$. This contradicts $\mathcal{A}_{i_1,\cdots,i_m}=0$ when $\{i_1,\cdots,i_m\}\in \bar{\mathbb{E}}$.
\end{proof}

From \cref{equ:ideal K}, $K$ is contained in the ideal
\[I_{\mathbb{E}}:=I((x_{i_1}\cdots x_{i_m})_{\{i_1,\cdots,i_m\}\in\bar{\mathbb{E}}}),\] which introduces an ideal-sparsity structure for \eqref{equ:dense GMP} inherited from the tensor $\mathcal{A}$. 
Recall the basic idea of ideal-sparsity is that, instead of optimizing over a single measure supported on $K\subset\mathbb{R}^n$, one may optimize over several measures that are supported on smaller dimensional spaces \cite{korda2024exploiting}. Here such smaller spaces are determined by the monomials $x_{i_1,\cdots,i_m}=0$.

\subsection{Ideal-sparse generalized moment problem reformulation}We abuse the maximal cliques to denote the maximal subsets $V_1,\cdots,V_p$ of $[n]$ that do not contain any set $\{i_1,\cdots,i_m\}$ with $x_{i_1}\cdots x_{i_m}=0$  $\{i_1,\cdots,i_m\in\bar{\mathbb{E}}\}$ in this subsection. From \cref{thm:maximal sets}, \cref{alg:clique} can generate these maximal subsets for any $\mathcal{A}\in S^m(\mathbb{R}^n)$.

Similar to Korda et al. \cite{korda2024exploiting}, we define the following subset of $K$:
\[\widehat{K_k}:=\{\mathbf{x}\in K:\text{Supp}(\mathbf{x})\subseteq V_k\}\subseteq K\subseteq \mathbb{R}^n~k\in[p].\]
And the projection of $\widehat{K_k}$ onto the subspace indexed by $V_k$ is defined as:
\begin{equation}\label{equ:ideal Kk}
K_k:=\{\mathbf{y}\in\mathbb{R}^{|V_k|}:\hat{\mathbf{y}}\in\widehat{K_k}\}\subseteq \mathbb{R}^{|V_k|}.
\end{equation}

In fact, for a nonnegative symmetric tensor $\mathcal{A}\in S^m(\mathbb{R}^n)$, if it is completely positive, its completely positive decomposition should inherit the clique property of its support multi-hypergraph.
\begin{lemma}\label{lem:clique decomposition}
	Let $G=(V,\mathbb{E})$ be the support multi-hypergraph of $\mathcal{A}\in S^m(\mathbb{R}^n)$. Suppose 
	\[\mathcal{A}=\sum_{r=1}^R(\mathbf{v}^r)^{\otimes m}\]
	is completely positive, $\mathbf{v}^1,\cdots \mathbf{v}^R$ are nonnegative. Then for each $\mathbf{v}^r,~r\in[R]$, there exists $k\in[p]$, such that
	\[\text{Supp}(\mathbf{v}^r)\subseteq V_k.\]
	That is, the support of a decomposition vector is contained in a maximal clique of $G=(V,\mathbb{E})$. 
\end{lemma}
\begin{proof}
	For any $\mathbf{v}^r$, $r\in[R]$, suppose that $\text{Supp}(\mathbf{v}^r)=\{i_1,\cdots,i_l\}$. Then for any $\{j_1,\cdots,j_m\}\subseteq \{i_1,\cdots,i_l\}$, $\mathcal{A}_{j_1,\cdots,j_m}\ge \mathbf{v}^r_{j_1}\cdots \mathbf{v}^r_{j_m}>0$, and $\{j_1,\cdots,j_m\}\in\mathbb{E}$. Hence, $\{i_1,\cdots,i_l\}$ is a clique and $\text{Supp}(\mathbf{v}^r)$ should be included in some maximal clique.
\end{proof}

\begin{theorem}\label{thm:equal2}
	Let $G=(V,\mathbb{E})$ be the support multi-hypergraph of $\mathcal{A}\in S^m(\mathbb{R}^n)$. $\widetilde{G}=(V,\widetilde{\mathbb{E}})$ is a multi-hypergraph such that $\mathbb{E}\subseteq\widetilde{\mathbb{E}}$. Let $\widetilde{V}_1,\cdots,\widetilde{V}_{\tilde{p}}$ denote the maximal cliques of $\widetilde{G}$, and define $\tilde{K}_k$ similarly as \cref{equ:ideal Kk}
	\[\tilde{K}_k:=\{\mathbf{y}\in\mathbb{R}^{|\widetilde{V}_k|}:\widehat{\mathbf{y}}^T\widehat{\mathbf{y}}-1=0,\widehat{y}_j\ge 0~j\in[n],\text{Supp}(\widehat{\mathbf{y}})\subseteq\widetilde{V}_k \}.\] $\mathcal{A}$ is complete positive if and only if the following generalized moment problem has a feasible solution $(\mu_1,\cdots,\mu_{\tilde{p}})$ with $\mu_k~(k\in[p])$ being finitely atomic.
	\begin{equation}\label{equ:sparse2}
		\begin{aligned}
			\widetilde{val}^{isp}:=\underset{\mu_k\in\mathcal{M}(\mathbb{R}^{|\tilde{V}_k|}),k\in[\tilde{p}]}{\min}&\sum_{k=1}^{\tilde{p}}\int F|_{\tilde{V}_k}d{\mu_k}\\
			\mbox{s.t.}&\sum_{k=1}^{\tilde{p}}\int \widehat{\mathbf{y}}^{\alpha}d{\mu_k}=\mathcal{A}_{i_1,i_2,\cdots,i_m}~(1\le i_1\le\cdots\le i_m\le n)\\
			&\text{Supp}(\mu_k)\subseteq \tilde{K}_k~(k\in[\tilde{p}]).
		\end{aligned}
	\end{equation}
\end{theorem}
\begin{proof}
	The sufficiency is obvious. Suppose that $(\mu_1,\cdots,\mu_{\tilde{p}})$ is a feasible solution of \cref{equ:sparse2}, and $\mu_k=\sum_{i=1}^{r_k}\lambda_{i}^k\delta_{\mathbf{y}^k_i}$, $\mathbf{y}^k_i\in \tilde{K}_k\subseteq\mathbb{R}^{|\widetilde{V}_k|}$, $\lambda_i\ge 0$. It leads that \[\mathcal{A}=\sum_{k=1}^{\tilde{p}}\sum_{i=1}^{r_k}\lambda_i^k(\widehat{\mathbf{y}}^k_i)^{\otimes m},\]
	where $\widehat{\mathbf{y}}^k_i$ is nonnegative and $\|\widehat{\mathbf{y}}^k_i\|=\|\mathbf{y}^k_i\|=1$. Hence, $\mathcal{A}$ is complete positive.
	
	If $\mathcal{A}$ is complete positive. Suppose that $\mathcal{A}=\sum_{r=1}^R\lambda_r(\mathbf{v}^r)^{\otimes m}$ with $\|\mathbf{v}^r\|=1$ and $\mathbf{v}^r~(r\in[R])$ are nonnegative. Then for any $\mathbf{v}^r$, $\mathbf{v}^r\in\widehat{K}_k$ follows from \cref{lem:clique decomposition}. That is, $\mathbf{v}^r\in K$ and $\text{Supp}(\mathbf{v}^r)\subseteq V_i$. Since each maximal clique of $G$ is contained in a maximal clique of $\widetilde{G}$, we have $\text{Supp}(\mathbf{v}^r)\in \tilde{V}_k$ for some $k\in[\tilde{p}]$. Then, 
	\[\mathcal{A}=\sum_{r=1}^R\lambda_r(\widehat{\mathbf{y}}^r)^{\otimes m},\]
	where for each $r\in [R]$, $\widehat{\mathbf{y}}^r=\mathbf{v}^r$ and $\mathbf{y}^r\in \tilde{K}_k$ for some $k$. Define the set
	\[\Lambda_k=\{\mathbf{x}\in \mathbb{R}^n:\text{Supp}(\mathbf{x})\subseteq \widetilde{V}_k, \text{Supp}(\mathbf{x})\not\subseteq \widetilde{V}_h \text{ for } 1\le h\le k-1\}~(k\in[\tilde{p}]).\] Consider \[\mu_k=\sum_{\{r:\widehat{\mathbf{y}}^r\in \Lambda_k\}}\lambda_r\delta_{{\mathbf{y}}^r}~k\in[\tilde{p}],\]
	we have $\text{Supp}(\mu_k)\subseteq \tilde{K}_k$ and $(\mu_1,\cdots,\mu_{\tilde{p}})$ is a feasible solution of \cref{equ:sparse2}.
\end{proof}

In particular, we define the following ideal-sparse generalized moment problem reformulation for \cref{equ:dense ideal}:
\begin{equation}\label{equ:sparse GMP}
	\begin{aligned}
		val^{isp}:=\underset{\mu_k\in\mathcal{M}(\mathbb{R}^{|V_k|}),k\in[p]}{\min}&\sum_{k=1}^p\int F|_{V_k}d{\mu_k}\\
		\mbox{s.t.}&\sum_{k=1}^p\int \widehat{\mathbf{y}}^{\bm{\alpha}}d{\mu_k}=\mathcal{A}_{i_1,i_2,\cdots,i_m}~(\{i_1,\cdots,i_m\}\in\mathbb{E})\\
		&\text{Supp}(\mu_k)\subseteq K_k~(k\in[p]).
	\end{aligned}
\end{equation}
Here, we use the superscript `isp' to indicate that the ideal-sparsity is exploited.
\begin{lemma}\label{lem:partite}
	$K=\widehat{K_1}\cup\cdots\cup\widehat{K_p}$
\end{lemma}
\begin{proof}
	Since $\widehat{K_k}\subseteq K$, $\widehat{K_1}\cup\cdots\cup\widehat{K_p}\subseteq K$. On the other hand, for any $x\in K$, $\text{Supp}(x)$ should not contain any set $\{i_1,\cdots,i_m\}$ with $\mathcal{A}_{i_1,\cdots,i_m}=0$. Otherwise, there exists $\{i_1,\cdots,i_m\}\subseteq \text{Supp(x)}$ such that $\{i_1,\cdots,i_m\}\in\bar{E}$ and $x_{i_1}\cdots x_{i_m}=0$. It leads that $x_{i_j}=0$ for some $i_j\in \text{Supp(x)}$, which conflicts. Thus, $\text{Supp}(x)\subseteq V_k$ and $x\in\widehat{K}_k$ for some $k\in[p]$. Therefore, $K\subseteq \widehat{K_1}\cup\cdots\cup\widehat{K_p}$. Hence, $K=\widehat{K_1}\cup\cdots\cup\widehat{K_p}$.
\end{proof}

We can also establish an equivalence between the dense problem \cref{equ:dense GMP} and the sparse reformuation \cref{equ:sparse GMP}, following the procedure similar to those of Korda et al.\cite{korda2024exploiting}. For completeness, we still provide a detailed proof for our case.
\begin{theorem}\label{thm:equal1}
	Problems \cref{equ:dense GMP} and \cref{equ:sparse GMP} are equivalent in the sense that their optimum values are equal, that is, $val=val^{isp}$.
\end{theorem}
\begin{proof}
	Since \eqref{equ:dense GMP} is equal to \eqref{equ:dense ideal}, we only need to prove the equivalence between \eqref{equ:dense ideal} and \eqref{equ:sparse GMP}. First, assume $(\mu_1,\cdots,\mu_p)$ is feasible for problem \eqref{equ:sparse GMP}. We can construct a measure $\mu\in\mathcal{M}(\mathbb{R}^n)$, defined by $\int fd\mu=\sum_{k=1}^p\int_{K_k}f|_{V_k}d\mu_k$ for any measurable function $f$ on $\mathbb{R}^n$. We have  
	\[\int fd\mu=\sum_{k=1}^p\int_{K_k}f|_{V_k}d\mu_k=\sum_{k=1}^p\int_{K_k}f|_{V_k}\chi^K|_{V_k}d\mu_k=\int f\chi^Kd\mu=\int_{K}fd\mu,\]
	where $\chi^K|_{V_k}(\mathbf{y})=\chi^K(\widehat{\mathbf{y}})=1$ for all $\mathbf{y}\in K_k$. Then, $\text{Supp}(\mu)\subseteq K$.
	In particular,
	\[\int \mathbf{x}^{\bm{\alpha}}d\mu=\sum_{k=1}^p\int_{K_k}\widehat{\mathbf{x}(V_k)}^{\bm{\alpha}}d\mu_k=\sum_{k=1}^p\int\widehat{\mathbf{y}}^{\bm{\alpha}}d\mu_k=\mathcal{A}_{i_1,\cdots,i_m}\]
	for any $\{i_1,\cdots,i_m\}\in \mathbb{E}$. As for $\{i_1,\cdots,i_m\}\in\bar{\mathbb{E}}$, $\int\mathbf{x}^{\bm{\alpha}}d\mu=\sum_{k=1}^p\int\widehat{\mathbf{y}}^{\bm{\alpha}}d\mu_k=0$. Hence, $\mu$ is feasible for \eqref{equ:dense ideal}.
	
	On the other hand, assume $\mu$ is feasible for \cref{equ:dense ideal}. For $k\in[p]$, define the set
	\[\Lambda_k=\{\mathbf{x}\in K:\text{Supp}(\mathbf{x})\subseteq V_k, \text{Supp}(\mathbf{x})\not\subseteq V_h \text{ for } 1\le h\le k-1\}\subseteq \widehat{K_k}.\]
	Since for any $\mathbf{x}\in K$, $\text{Supp}(\mathbf{x})\subseteq V_k$ for some $k\in[p]$ according to \cref{lem:partite}, the sets $\Lambda_1,\cdots,\Lambda_p$ form a disjoint partition of $K$. We can construct the measure $\mu_k\in\mathcal{M}(\mathbb{R}^{|V_k|})$, given by $\int fd\mu_k=\int_{\Lambda_k}f(\mathbf{x}(V_k))d\mu$ for any measurable function $f$ on $\mathbb{R}^{|V_k|}$. We have 
	\[\int fd\mu_k=\int_{\Lambda_k}f(\mathbf{x}(V_k))d\mu=\int_{\Lambda_k}f(\mathbf{x}(V_k))\chi^{K_k}(\mathbf{x}(V_k))d\mu=\int f\chi^{K_k}d\mu_k=\int_{K_k}fd\mu_k,\]
	then $\text{Supp}(\mu_k)\subseteq K_k$. We also have
	\[\sum_{k=1}^p\int\widehat{\mathbf{y}}^{\bm{\alpha}}d\mu_k=\sum_{k=1}^p\int_{\Lambda_k}\widehat{\mathbf{x}(V_k)}^{\bm{\alpha}}d\mu=\sum_{k=1}^p\int_{\Lambda_k}\mathbf{x}^{\bm{\alpha}}d\mu=\int_{K}\mathbf{x}^{\alpha}d\mu=\mathcal{A}_{i_1,\cdots,i_m}\]
	for any $\{i_1,\cdots,i_m\}\in\mathbb{E}$, where the third equality follows from that $\Lambda_1,\cdots,\Lambda_p$ disjointly partition $K$. Hence, $(\mu_1,\cdots,\mu_k) $ is feasible for \cref{equ:sparse GMP}.
\end{proof}

\subsection{Ideal-sparse moment relaxations}
This subsection studies the ideal-sparse moment relaxations for solving the ideal-sparse generalized moment problem reformulation.  
Based on the ideal-sparse reformulation \cref{equ:sparse GMP}, we can define the ideal-sparse moment hierarchies for \cref{equ:dense GMP}:
\begin{equation}\label{equ:sparse relax}
	\begin{aligned}
		\xi_t^{isp}:=\underset{\mathbf{z}_k\in\mathbb{R}^{\mathbb{N}^{|V_k|}_{2t}},k\in[p]}{\min}&\sum_{k=1}^p\langle F|_{V_k},\mathbf{z}_k\rangle\\
		\mbox{s.t.}&\sum_{\{k\in[p]:i_j\in V_k,j\in[m]\}}\langle\widehat{\mathbf{y}}^{\bm\alpha},\mathbf{z}_k\rangle=\mathcal{A}_{i_1,i_2,\cdots,i_m}~(\{i_1,\cdots,i_m\}\in \mathbb{E})\\
		&M_t(\mathbf{z}_k)\succeq 0~(k\in[p])\\
		&L_{x_j}^{(t)}(\mathbf{z}_k)\succeq 0~(j\in V_k,k\in[p])\\
		&L_{\mathbf{x}(V_k)^T\mathbf{x}(V_k)-1}^{(t)}(\mathbf{z}_k)=0, ~(k\in[p]).
	\end{aligned}
\end{equation}
And the dual optimization of \eqref{equ:sparse relax} is 
\begin{equation}\label{equ:sparse dual}
	\begin{aligned}
		\underset{P_k\in\mathbb{R}[\mathbf{x}(V_k)]_{\mathcal{A}}}{\max}&\sum_{k=1}^p\langle P_k,\mathbf{w}\rangle\\
		\text{s.t.}&F|_{V_k}-P_k\in I((\mathbf{x}(V_k)^T\mathbf{x}(V_k)-1))_{2t}+Q((x_j)_{j\in V_k})_{t}, k\in[p],
	\end{aligned}
\end{equation}
where $\mathbb{R}[\mathbf{x}(V_k)]_{\mathcal{A}}$ denotes the set of polynomials in the form of $\sum_{\bm{\alpha}}P_{\bm{\alpha}}\widehat{\mathbf{x}(V_k)}^{\bm{\alpha}}$ and  $\mathbf{w}$ is the same as defined in \cref{equ:dense dual}. 

We refer to the above hierarchy of parameter $\xi_t^{isp}$ as ideal-sparse moment hierarchy.
Hence, for the completely positive tensor decomposition problem, if the tensor $\mathcal{A}\in S^n(\mathbb{R}^n)$ is sparse, the computation of ideal-sparse moment hierarchys \cref{equ:sparse relax} will be much faster than \cref{equ:dense relax}, since $|V_k|$ may be much smaller than $n$. Let $\Omega(\xi_t)$ and $\Omega(\xi_t^{isp})$ denote the feasible sets of \cref{equ:dense relax} and \cref{equ:sparse relax}, respectively. It is obvious that $\Omega(\xi_t^{isp})\subseteq \Omega(\xi_t)$ as stated in the following lemma. 
\begin{lemma}\label{lem:include}
	If $(\mathbf{z}_1,\cdots,\mathbf{z}_p)$ is feasible for \cref{equ:sparse relax}, then $\mathbf{z}\in\mathbb{R}^{\mathbb{N}^n_{2t}}$ is feasible for $\cref{equ:dense relax}$ with the same objective value, given by $\mathbf{z}_{\bm{\alpha}}=\sum_{k=1}^p\langle \mathbf{x}^{\bm{\alpha}}|_{V_k}, \mathbf{z}_k\rangle$. Hence, $\xi_t\le\xi_t^{isp}$.
\end{lemma}
\begin{proof}
	From the construction of $\mathbf{z}$, we have for any $\{i_1,\cdots,i_m\}\in \mathbb{E}$, 
	\[\mathbf{z}_{\bm{\alpha}}=\sum_{k=1}^p\langle \mathbf{x}^{\bm{\alpha}}|_{V_k}, \mathbf{z}_k\rangle=\sum_{\{k\in[p]:i_j\in V_k,j\in[m]\}}\langle\widehat{\mathbf{y}}^{\bm\alpha},\mathbf{z}_k\rangle=\mathcal{A}_{i_1,i_2,\cdots,i_m}.\]
	And $\mathbf{z}_{\bm{\alpha}}=0$ for any $\{i_1,\cdots,i_m\}\in\bar{\mathbb{E}}$, since $V_k$ contains no such $\{i_1,\cdots,i_m\}$ for any $k\in[p]$. We also have for any $q\in\mathbb{R}[\mathbf{x}]_{2t}$ and $p\in\mathbb{R}[\mathbf{x}]_{t-\lceil\text{deg}(q)/2\rceil}$,
	\begin{equation*}
		\begin{aligned}
	vec(p)^TL^{(t)}_{q}(\mathbf{z})vec(p)&=\langle qp^2,\mathbf{z}\rangle\\
	&=\sum_{k=1}^p\langle qp^2|_{V_k},\mathbf{z}_k\rangle=\sum_{k=1}^p vec(p|_{V_k})^TL^{(t)}_{q|_{V_k}}(\mathbf{z}_k)vec(p|_{V_k}).
	\end{aligned}
	\end{equation*}
Then it follows that $\mathbf{z}$ is peasible for \cref{equ:dense relax} and $\xi_t\le \xi_t^{isp}$.
\end{proof}

\begin{example}
	Consider the tensor $\mathcal{A}\in S^3(\mathbb{R}^3)$ given in \cref{exa:example}. The maximal cliques are $V_1=\{1,3\}$ and $V_2=\{1,2\}$. Suppose that the relaxation level $t=2$ and $(\mathbf{z}_1,\mathbf{z}_2)$ is feasible for \cref{equ:sparse relax} with $\mathbf{z}_1,\mathbf{z}_2\in\mathbb{R}^{\mathbb{N}^2_4}$. Then a feasible point $\mathbf{z}\in\mathbb{R}^{\mathbb{N}^3_4}$ for \cref{equ:dense relax} can be constructed as 
	\[z_{300}= (\mathbf{z}_1)_{30}+(\mathbf{z}_2)_{30},z_{210}=(\mathbf{z}_2)_{21},z_{201}=(\mathbf{z}_1)_{21},z_{120}=(\mathbf{z}_2)_{12},z_{111}=0,\]
	\[z_{102}=(\mathbf{z}_1)_{12},z_{030}=(\mathbf{z}_2)_{03},z_{021}=0,z_{012}=0,z_{003}=(\mathbf{z}_1)_{03}.\]
	Here we only list the elements of $\mathbf{z}$ corresponding to $\mathcal{A}$.
\end{example}

\begin{remark}
	Korda et al. \cite{korda2024exploiting} have pointed out that one possible drawback of the ideal-sparse moment relaxation \cref{equ:sparse relax} is that the number of maximal cliques of graph $G$ could be large. Then one can consider the extension $\widetilde{G}$ of $G$. Following this, we can define the corresponding ideal-sparse moment relaxation for \cref{equ:sparse2}:
	\begin{equation}\label{equ:sparse relax2}
		\begin{aligned}
			\widetilde{\xi}_t^{isp}:=\underset{\mathbf{z}_k\in\mathbb{R}^{\mathbb{N}^{|\widetilde{V}_k|}_{2t}},k\in[\tilde{p}]}{\min}&\langle \sum_{k=1}^pF|_{\widetilde{V}_k},\mathbf{z}_k\rangle\\
			\mbox{s.t.}&\sum_{\{k\in[p]:i_j\in \widetilde{V}_k,j\in[m]\}}\langle\widehat{\mathbf{y}}^{\bm\alpha},\mathbf{z}_k\rangle=\mathcal{A}_{i_1,i_2,\cdots,i_m}~(1\le i_1\le\cdots\le i_m)\\
			&M_t(\mathbf{z}_k)\succeq 0~(k\in[\tilde{p}])\\
			&L_{x_j}^{(t)}(\mathbf{z}_k)\succeq 0~(j\in \widetilde{V}_k,k\in[\tilde{p}])\\
			&L_{\mathbf{x}(\widetilde{V}_k)^T\mathbf{x}(\widetilde{V}_k)-1}^{(t)}(\mathbf{z}_k)=0, ~(k\in[\tilde{p}]).
		\end{aligned}
	\end{equation}
\cref{equ:sparse relax2} only partially expolits the ideal sparsity, while involves less cliques but with larger sizes. However, for a support multi-hypergraph, it may be necessary to first address how to obtain the extension such as a chordal extension. We leave it for further research. Nonetheless, in our numerical experiments (see \cref{sec:numerical} please), it seems that the number of maximal cliques of \cref{equ:sparse relax} is acceptable, considering the advantages of fully exploiting the sparsity.
\end{remark}

Let $t_0=\max\{\lceil \text{deg}(F)/2\rceil,\lceil m/2\rceil,1\}$ and set $d_K=\max\{1,\lceil \text{deg}(x_j)/2\rceil~(j\in[n]),\lceil\text{deg}(\mathbf{x}^T\mathbf{x}-1)/2\rceil\}$. Similar to \cref{thm:flat}, if the optimal solution $(\mathbf{z}_1,\cdots,\mathbf{z}_p)$ of the ideal-sparse moment relaxation \eqref{equ:sparse relax} satisfies the following flatness condition:
\begin{equation}\label{equ:multi-rank}
	rank~M_{s_k}(\mathbf{z}_k)=rank~M_{s_k-d_K}(\mathbf{z}_k)=:r_k
\end{equation} 
for some integer $s_k$ such that $t_0\le s_k\le t$ for each $k\in[p]$. Then, equality $\xi_t^{isp}=val^{isp}$ holds, and the reformulation \eqref{equ:sparse GMP} has an optimal solution $(\mu_1,\cdots,\mu_p)$, where each $\mu_k$ is $r_k$-atomic.

Based on the flatness condition \cref{equ:multi-rank}, we can obtain the convergence theorem and properties of ideal-sparse moment relaxations \cref{equ:sparse relax} similar to those of dense moment relaxations \cref{equ:dense relax}. The convergence is obtained under the following assumption. Please refer to \cite{nie2014truncated} regarding the generic property.
\begin{assumption}\label{assump:F}
	The randomly generated $F$ of \cref{equ:sparse relax} is generic in $\Sigma[\mathbf{x}]_d$, $F\in int(\Sigma[\mathbf{x}]_d)$ and $F|_{V_k}\in int(\Sigma[\mathbf{x}(V_k)]_d)$ for $k\in[p]$.
\end{assumption}
\begin{theorem}[Convergence analysis]\label{thm:checkable}
	Suppose $\mathcal{A}\in S^m(\mathbb{R}^n)$. Let $d>m$, the ideal-sparse moment hierarchies has the following properties:
	\begin{enumerate}
		\item If \cref{equ:sparse relax} is infeasible for some $t$, then $\mathcal{A}$ is not completely positive.
		\item If $\mathcal{A}$ is not completely positive, then \cref{equ:sparse relax} is infeasible for all $t$ big enough.
		\item If $\mathcal{A}$ is completely positive and \cref{assump:F} holds, \cref{equ:sparse relax} has an optimizer $(\mathbf{z}^{*,t}_1,\cdots,\mathbf{z}^{*,t}_p)$ for all $t\ge d/2$. 
		\item For all $s$ big enough, the sequence $\{\mathbf{z}^{*,t}|_{2s}\}_t$, where $\mathbf{z}^{*,t}_{\bm{\alpha}}=\sum_{k=1}^p\langle \mathbf{x}^{\bm{\alpha}}|_{V_k}, \mathbf{z}_k^{*,t}\rangle$, is bounded and all its accumulation points are flat.
		\item If \cref{equ:dense relax} is finite convergent, then \cref{equ:sparse relax} is also finite convergent. That is, \cref{equ:sparse relax} has an optimizer $(\mathbf{z}^{*,t}_1,\cdots,\mathbf{z}^{*,t}_p)$ satisfying the flatness condition \cref{equ:multi-rank} for $t$ big enough.
	\end{enumerate}
\end{theorem}
 
\begin{proof}
		1. If \cref{equ:sparse relax} is infeasible for some $t$, suppose that $\mathcal{A}$ is completely positive, then \cref{equ:dense GMP} is feasible. Since \cref{equ:sparse relax} is a relaxation of \cref{equ:dense ideal}, \cref{equ:sparse relax} is also feasible, which contradicts.
		
		2. If $\mathcal{A}$ is not completely positive, then \cref{equ:dense relax} is infeasible for all $t$ big enough from \cref{thm:nie}. Suppose that \cref{equ:sparse relax} has a feasible point $(\mathbf{z}_1,\cdots,\mathbf{z}_p)$. Define $\mathbf{z}\in\mathbb{R}^{\mathbb{N}^n_{2t}}$ by setting $\mathbf{z}_{\bm{\alpha}}=\sum_{k=1}^p\langle \mathbf{x}^{\bm{\alpha}}|_{V_k}, \mathbf{z}_k\rangle$, then $\mathbf{z}$ is feasible for \cref{equ:dense relax} according to \cref{lem:include}. It contradicts the assumption. Therefore, \cref{equ:sparse relax} should also be infeasible for all $t$ big enough.
		
		3. From 1, \cref{equ:sparse relax} has a feasible point $(\mathbf{z}_1,\cdots,\mathbf{z}_p)$ when $\mathcal{A}$ is completely positive. Since $F|_{V_k}\in int(\Sigma[\mathbf{x}(V_k)]_d)$, the tuple of zero polynomials is a strictly feasible point of the dual problem \cref{equ:sparse dual} of \cref{equ:sparse relax}. That is, the Slater condition is satified. Hence, the strong duality holds and \cref{equ:sparse relax} has an optimizer $(\mathbf{z}_1^{*,t},\cdots,\mathbf{z}_p^{*,t})$ for every $t\ge d/2$.
		
		4. Since $F|_{V_k}\in int(\Sigma[\mathbf{x}(V_k)]_d)$, there exists $\epsilon>0$ such that $F|_{V_k}-\epsilon\in\Sigma[\mathbf{x}(V_k)]_d $. So, we have
		\[0\le\langle F|_{V_k}-\epsilon, \mathbf{z}_k^{*,t}\rangle=\langle F|_{V_k}, \mathbf{z}_k^{*,t}\rangle- \epsilon\langle 1, \mathbf{z}_k^{*,t}\rangle.\]
		It implies that
		\[(\mathbf{z}_k^{*,t})_0=\langle 1, \mathbf{z}_k^{*,t}\rangle\le\langle F|_{V_k}, \mathbf{z}_k^{*,t}\rangle/\epsilon\le \sum_{k=1}^p \langle F|_{V_k}, \mathbf{z}_k^{*,t}\rangle/\epsilon=\xi_t^{isp}/\epsilon\le val^{isp}/\epsilon.\]
		Since $L_{\mathbf{x}^T(V_k)\mathbf{x}(V_k)-1}^{(t)}(\mathbf{z}_k^{*,t})=0$, $\|\mathbf{x}(V_k)\|^{2s}\in \Sigma[\mathbf{x}(V_k)]_d$, we have
		\[\langle \|\mathbf{x}(V_k)\|^{2s}, \mathbf{z}_k^{*,t}\rangle=\langle \|\mathbf{x}(V_k)\|_2^{2s-2}, \mathbf{z}_k^{*,t}\rangle=(\mathbf{z}_k^{*,t})_0,~s=1,\cdots,t.\]
		Since $M_t(\mathbf{z}_k^{*,t})\succeq 0$, for each $s=0,1,\cdots,t$ we have
		\begin{equation*}
			\begin{aligned}
				\|\mathbf{z}_k^{*,t}|_{2s}\|&\le \|M_s({\mathbf{z}_k^{*,t}})\|_F\le tr(M_s({\mathbf{z}_k^{*,t}}))=\sum_{i=0}^s\sum_{|\bm{\alpha}|=i}(\mathbf{z}_k^{*,t})_{2\bm{\alpha}}=\sum_{i=0}^s\sum_{|\bm{\alpha}|=i}\langle \mathbf{x}(V_k)^{2\bm{\alpha}},\mathbf{z}_k^{*,t}\rangle\\
				&\le\sum_{i=0}^s\langle \|\mathbf{x}(V_k)\|^{2i},\mathbf{z}_k^{*,t}\rangle\le \sum_{i=0}^s(\mathbf{z}_k^{*,t})_0\le (s+1)val^{isp}/\epsilon\le (s+1)C/\epsilon,  
			\end{aligned}
		\end{equation*}
	where $\bm{\alpha}\in\mathbb{N}^{|V_k|}_t$ and $C$ is the objective value of \cref{equ:sparse GMP} with the feasible value corresponding to $\mathcal{A}$.
	Hence, the sequence $\{(\mathbf{z}^{*,t}_1|_{2s},\cdots,\mathbf{z}^{*,t}_p|_{2s})\}_t$ is bounded, and so is $\{\mathbf{z}^{*,t}|_{2s}\}_t$. They must have convergent subsequences. We can generally assume $\{(\mathbf{z}^{*,t}_1|_{2s},\cdots,\mathbf{z}^{*,t}_p|_{2s})\}_t$ converges to $(\bm{\omega}_1,\cdots,\bm{\omega}_k)$ and $\{\mathbf{z}^{*,t}|_{2s}\}_t$ converges to $\bm{\omega}$. 
	
	According to the proof of \cite[Theorem 5.3(i)]{nie2014truncated}, there exists $(\mathbf{z}^*_1,\cdots,\mathbf{z}^*_p)$ such that
	$\mathbf{z}^*_k|_{2s}=\bm{\omega}_k$ and $\mathbf{z}^*_k$ admits a measure $\mu_k^*$ supported in $K_k$. Hence, $(\mu_1^*,\cdots,\mu_p^*)$ is feasible for \cref{equ:sparse GMP}, and
	\begin{equation*}
		\begin{aligned} 
	val^{isp}&\le \sum_{k=1}^p \int F|_{V_k}d\mu^*_k=\sum_{k=1}^p\langle F|_{V_k},\mathbf{z}^*_k|_d\rangle=\sum_{k=1}^p\langle F|_{V_k},\bm{\omega}_k|_{d}\rangle\\
	&=\lim_{k\rightarrow\infty}\sum_{k=1}^p\langle F|_{V_k},\mathbf{z}^{*,t}_k\rangle\le val^{isp}.
	\end{aligned} 
	\end{equation*}
	Therefor, $(\mu_1^*,\cdots,\mu_p^*)$ is a minimizer of \cref{equ:sparse GMP}. Then, there exists $\mathbf{z}^{*}$ such that $\mathbf{z}^*|_{2s}=\bm{\omega}$. And it admits a measure $\mu^*$ supported in $K$, given by $\int fd\mu^*=\sum_{k=1}^p\int_{K_k}f|_{V_k}d\mu_k^*$ for any measurable function $f$ on $\mathbb{R}^n$. $\mu_*$ is also a minimizer of \cref{equ:dense GMP} from \cref{thm:equal1}. Through \cite[Proposition 5.2]{nie2014truncated} and \cite[lemma 3.5]{nie2014truncated}, \cref{equ:truncated} has a unique minimizer admiting the finitely atomic measure when $F$ is generic in $\Sigma[\mathbf{x}]_d$, which must be $\bm{\omega}|_d$. Again from the proof of \cite[Theorem 5.3(i)]{nie2014truncated}, $\bm{\omega}$ is flat for $s$ big enough .
	
	5. Suppose that $\mathbf{z}^{*,\hat{t}}$ is a minimizer of \cref{equ:dense relax} and $\mathbf{z}^{*,\hat{t}}|_{2s}$ is flat for some $s\ge d/2$. Then $\mathbf{z}^{*,\hat{t}}|_{2s}$ admits a unique finitely atomic measure $\mu=\sum_{r=1}^R\lambda_r\delta_{\mathbf{v}_r}$ supported in $K$ by \cite[Theorem 2.2]{nie2014truncated}, with $R=rank~ M_s(\mathbf{z}^{*,\hat{t}}|_{2s})$. It implies that $\mathcal{A}=\sum_{r=1}^R\lambda_r(\mathbf{v}_r)^{\otimes m}$. For each $\mathbf{v}_r$, $\text{Supp}(\mathbf{v}_r)\subseteq V_k$ for some $k\in[p]$ by \cref{lem:clique decomposition}. Let $$\mu_k=\sum_{s\in S_k}\lambda_s\delta_{\mathbf{v}_s(V_k)},$$
	where 
	\[S_k=\{s:\text{Supp}(\mathbf{v}_s)\subseteq V_k, \text{Supp}(\mathbf{v}_s)\not\subseteq V_h \text{ for }1\le h\le k-1\}~k\in[p],\]
	then $(\mu_1,\cdots,\mu_p)$ is feasible for \cref{equ:sparse GMP}. And for any $t\ge \hat{t}$, $(\mathbf{z}^{*,t}_1,\cdots,\mathbf{z}^{*,t}_p)$ is feasible for \cref{equ:sparse relax}, where $$\mathbf{z}^{*,t}_k=\int[\mathbf{x}(V_k)]_{2t}d\mu_k=\sum_{s\in S_k}\lambda_s[\mathbf{v}_s(V_k)]_{2t}.$$ Moreover, the objective values of \cref{equ:dense relax} and \cref{equ:sparse relax} are equal to $val$ in this case. Thus, $(\mathbf{z}^{*,t}_1,\cdots,\mathbf{z}^{*,t}_p)$ is a minimizer of \cref{equ:sparse relax} from \cref{lem:include} and each $\mathbf{z}^{*,t}_k$ admits a finitely atomic measure $\mu_k$ for $t\ge \hat{t}$. 
	
	We now prove that if $t\ge (R+1)d_K$, then $(\mathbf{z}_1^{*,t},\cdots,\mathbf{z}_p^{*,t})$ satisfies the flatness condition \cref{equ:multi-rank}. Since
	\begin{equation*}
		\begin{aligned}
	rank~ M_0(\mathbf{z}_k^{*,t})&\le rank~ M_{d_K}(\mathbf{z}_k^{*,t})\le\cdots\le rank~ M_{(R+1)d_K}(\mathbf{z}_k^{*,t})\\
	&\le |\text{Supp}(\mu_k)|\le R,
	\end{aligned}
	\end{equation*}
there exists $l\ge R+1$, such that $rank~ M_{(l-1)d_K}(\mathbf{z}_k^{*,t})=rank~ M_{ld_K}(\mathbf{z}_k^{*,t})$ for $k\in[p]$. Hence, the flatness condition \cref{equ:multi-rank} holds.
\end{proof}
\begin{remark}
	In our numerical experiments, the finite convergence is observed.
\end{remark}

\section{Numeraical experiments}\label{sec:numerical}
In this section, we present numerical experiments that apply the ideal-sparse generalized moment problem reformulation to solve the completely positive tensor decomposition problem. We first explore the behaviour of the maximal clique generation \cref{alg:clique} on randomly generatred sparse tensors. Then dense and sparse moment relaxations are applied to different tensors from the literature, to illustrate the superiority of the ideal-sparse hierarchies compared to the dense ones. A comparison with TSSOS \cite{wang2021tssos,wang2022cs} exploiting both correlative and term sparsity in the dual problem \cref{equ:dense dual} is also conducted.

All computations were implemented on a personal computer running Windows 11 with Gen Intel Core i7 at 2.3GHz and 32GB memory. The ideal-dense and sparse model were implemented in Julia \cite{bezanson2017julia} utilizing MOSEK to slove the semidefinite programming (SDP).
\subsection{Check randomly generated sparse tensor by \cref{alg:clique}}\label{subsec:clique} 
According to \cref{thm:necessary}, before utilizing the moment relaxation method, one can first employ \cref{alg:clique} to check whether a symmetric tensor is completely positive, especially when it is sparse. We first describe how we construct random sparse symmetric tensors. Given integers $n\in\mathbb{N}$ and $d\in\mathbb{N}$. Denote the non-zero density by $nzd$, since the maximal cliques generated by \cref{alg:clique} only determined by the position of zero entries, we create a symmetric binary tensor $\mathcal{A}\in S^m(\mathbb{R}^n)$ with the diagonal elements are ones. And there are $\lceil{nzd\cdot ({n+m-1 \choose m}-n)}\rceil+n$ ones in $\mathbf{w}$, given by $w_{\bf{\alpha}}=\mathcal{A}_{i_1,\cdots,i_m}$, whose positions are selected uniformly at random.  

We generate such random examples for varying size $(n=10,12,14 \text{ and }m=4,6,8)$. The non-zero density varies from 0.4 to 0.98. And 5 examples are generated per size and nzd value. The maximal computation time of checking the necessary condition by \cref{alg:clique} is shown in \cref{tab:clique}. For clarity, we only keep the decimal digits according to the order of magnitude. Interestingly, most of the sparse tensors randomly generated in our tests are not completely positive and can be distinguished solely through \cref{alg:clique}. An intuitive explanation might be due to the zero-dominace property. Since for a completely positive tensor $\mathcal{A}$, if $\mathcal{A}_{i_1,\cdots,i_m}\neq 0$, then any $\mathcal{A}_{j_1,\cdots,j_m}$ with $[\{j_1,\cdots,j_m\}]\subseteq [\{i_1,\cdots,i_m\}]$ should not be zero. This is hard to satisfy in a random generated sparse symmetric tensor.

\cref{tab:clique} indeed shows the efficiency of \cref{alg:clique} and its advatage in distinguishing sparse tensors that are not completely positive. Another observation is that the computation time grows as the dimension and order increase. However, the growth rate is more acceptable than that of the dense moment relaxations.
\begin{table}[H]\label{tab:clique}
	\caption{Check the necessary condition of \cref{thm:necessary} by \cref{alg:clique}}
	\begin{center}
	\begin{tabular}{cc|cll|cll|cll}
		\hline
		\multirow{2}{*}{$p$}     & $n$                   & \multicolumn{3}{c|}{10}                                                            & \multicolumn{3}{c|}{12}                                                         & \multicolumn{3}{c}{14}                                                         \\ \cline{2-11} 
		& $d$                   & 4                          & \multicolumn{1}{c}{6}     & \multicolumn{1}{c|}{8}    & 4                         & \multicolumn{1}{c}{6}    & \multicolumn{1}{c|}{8}   & 4                         & \multicolumn{1}{c}{6}    & \multicolumn{1}{c}{8}   \\ \hline
		0.4                      &                       & 0.0006                     & \multicolumn{1}{c}{0.005} & \multicolumn{1}{c|}{0.05} & 0.001                     & \multicolumn{1}{c}{0.01} & \multicolumn{1}{c|}{0.1} & 0.002                     & \multicolumn{1}{c}{0.07} & \multicolumn{1}{c}{0.4} \\
		\multicolumn{1}{l}{0.8}  & \multicolumn{1}{l|}{} & \multicolumn{1}{l}{0.0005} & 0.005                     & 0.03                      & \multicolumn{1}{l}{0.001} & 0.01                     & 0.1                      & \multicolumn{1}{l}{0.003} & 0.07                     & 0.3                     \\
		\multicolumn{1}{l}{0.98} & \multicolumn{1}{l|}{} & \multicolumn{1}{l}{0.0008} & 0.006                     & 0.03                      & \multicolumn{1}{l}{0.003} & 0.05                     & 0.2                      & \multicolumn{1}{l}{0.007} & 0.07                     & 0.4                     \\ \hline
	\end{tabular}
\end{center}
\end{table}

\subsection{Comparison with the dense model}
The following example is presented to see whether the proposed hierarchies are able to detect that the tensor is not completely positive. This can be achieved when the solver returns an infeasibility certificate.
\begin{example}\label{exa:noncp}
	 We will consider the following two tensors that are known to be symmetric nonnegative by not completely positive:	
	\begin{enumerate}[(1)]
		\item non\_ex1: $\mathcal{A}\in S^3(\mathbb{R}^{11})$ given as Example 4.1 of Fan et al. \cite{fan2017semidefinite}.
		\item non\_ex2: $\mathcal{A}\in S^5(\mathbb{R}^8)$ given as Example 4.2 of Fan et al. \cite{fan2017semidefinite}.
	\end{enumerate}
\end{example}
The numerical results for \cref{exa:noncp} are presented in \cref{tab:noncp}. We also list the tensor dimension $n$, order $m$ and relaxation level $t$ in the table. Here ``Gloptipoly" represents the implementaion of the dense moment relaxations in MATLAB 2022a using the software GlotiPoly 3 \cite{henrion2009gloptipoly} with SeDuMi as the SDP solver, which is the same as Fan et al. \cite{fan2017semidefinite}. $\xi_t$ and $\xi_t^{isp}$ denote the dense and ideal-sparse moment relaxation of level $t$. Since solving the SDP relaxation takes up the dominant computation time, we only show the times solving the SDP relaxation. The total CPU times for the ideal-sparse method including the modeling process are recorded in \cref{tab:tssos}.

From the results, we observe that ideal-sparse moment relaxations perferom much better than the dense ones as it returns the infeasibility certificate using less run time.
\begin{table}[H]\label{tab:noncp}
	\caption{Solving the SDP of the dense and ideal-sparse moment relaxations for \cref{exa:noncp}}
	\begin{center}
	\begin{tabular}{lllllll}
		\hline
		$\mathcal{A}$ & $n$ & $m$ & $t$ & Gloptipoly      & $\xi_t$ & $\xi_t^{isp}$  \\ \hline
		non\_ex1      & 11  & 3   & 2   & 1.00      & 1.15   &$< 0.01$ \\
		non\_ex2      & 8   & 5   & 3   & 5.1515 & 66.87 & $< 0.01$ \\ \hline
	\end{tabular}
\end{center}
\end{table}
\begin{example}\label{exa:cp}
	We will consider the following seven example tensors, which are known to be completely positive in the literature:
	\begin{enumerate}[(1)]
		\item ex1: \cref{exa:example}
		\item ex2: $\mathcal{A}\in S^3(\mathbb{R}^{10})$ given as Example 4.3 of Fan et al. \cite{fan2017semidefinite}.
		\item ex3: $\mathcal{A}\in S^4(\mathbb{R}^{10})$ given as Example 4.4 of Fan et al. \cite{fan2017semidefinite}.
		\item ex4: $\mathcal{A}\in S^3(\mathbb{R}^{10})$ given as the $m=3,~n=10$ case (1) in Example 2 of \cite{qi2014nonnegative}.
		\item ex5: $\mathcal{A}\in S^3(\mathbb{R}^{10})$ given as the $m=3,~n=10$ case (2) in Example 2 of \cite{qi2014nonnegative}.
		\item ex6: $\mathcal{A}\in S^4(\mathbb{R}^{10})$ given as the $m=4,~n=10$ case (2) in Example 2 of \cite{qi2014nonnegative}.
		\item ex7: $\mathcal{A}\in S^4(\mathbb{R}^{10})$ given as the $m=4,~n=10$ case (3) in Example 2 of \cite{qi2014nonnegative}.
	\end{enumerate}
\end{example}
We first show in \cref{tab:mc} the maximal cliques (denoted by mc in the table) of support multi-hypergraph $G=([n],\mathbb{E})$ the completely positive tensors in \cref{exa:cp} generated by \cref{alg:clique}. The correctness of the generated maximal cliques can be validated through the consistency with support sets of decomposition vectors given in literature.

It also illustrates that the positive semidifinite (PSD) matrix size in the ideal-sparse model can be significantly smaller than in the dense one. For instance, if we formulate a dense model for ex7 with relaxation order $t_0=\lceil(m+1)/2\rceil=3$, than the maximal PSD matrix size is ${n+t_0 \choose t_0}\times{n+t_0 \choose t_0}$, that is a 286 by 286 matrix. While with the ideal-sparsity, the maximal PSD matrix size involved is reduced to ${4+t_0\choose t_0}\times {4+t_0\choose t_0}$, that is a 35 by 35 matrix. Although there are seven maximal cliques, which lead to more matrices in the ideal-sparse moment relaxation, the problem scale is still smaller than the dense one. 
\begin{table}[H]\label{tab:mc}
	\caption{The maximal cliques of tensors in \cref{exa:cp}}
	\begin{center}
	\begin{tabular}{llll}
		\hline
		$\mathcal{A}$ & $n$ & $m$ & mc                                                                                            \\ \hline
		ex1           & 3   & 3   & \{1,2\},\{1,3\}                                                                           \\
		ex2           & 10  & 3   & \{2,3,8\},\{2,4,5\},\{3,4,5\},\{3,7,8\},\{4,7,9\},\{4,10\},\{8,9,10\}           \\
		ex3           & 10  & 4   & \{1,10\},\{2,4,8,9\},\{5,7,9\},\{6,7,9\},\{6,10\},\{8,9,10\}                      \\
		ex4           & 10  & 3   & \{1,5\},\{2,3\},\{2,6,9\},\{2,8,10\},\{3,4,5\},\{5,9\},\{7,9,10\}               \\
		ex5           & 10  & 3   & \{1,5,10\},\{2,3,9\},\{2,9,10\},\{2,8\},\{3,4,8\},\{5,8\}                         \\
		ex6           & 10  & 4   & \{1,2,7\},\{1,3,8,9\},\{1,5\},\{2,6,7\},\{2,3,6\},\{4,9\},\{7,9,10\}            \\
		ex7           & 10  & 4   & \{1,5,9,10\},\{1,5,6,9\},\{1,5,6,8\},\{2,5,6,9\},\{3,4,9\},\{3,9,10\},\{5,8,9\} \\ \hline
	\end{tabular}
\end{center}
\end{table}

The SDP computation times for \cref{exa:cp} are presented in \cref{tab:cp}, denoted by `SDP (s)' in the table. The total CPU times of \cref{exa:cp} for the ideal-sparse method including the modeling process are also recorded in \cref{tab:tssos}. `-' stands for that MOSEK throws an out of space error during solving the SDP. The superiority of the ideal-sparse model is also confirmed. 

\begin{table}[H]\label{tab:cp}
	\caption{Comparison with the dense model for \cref{exa:cp}}
	\begin{tabular}{llllllllll}
		\hline
		$\mathcal{A}$ & $n$ & $m$ & $t$ & Gloptipoly    &          & $\xi_t^{cp}$  &          & $\xi_t^{isp}$   &          \\ \cline{5-10} 
		&     &     &     & SDP (s) & Flat(1) & SDP (s) & Flat(1) & SDP (s)   & Flat(2) \\ \hline
		ex1           & 3   & 3   & 2   & 0.03          & True     & 0.01          & True     & \textless{}0.01 & True     \\
		ex2           & 10  & 3   & 2   & 1.86          & False    & 1.36          & False    & 0.03            & True     \\
		ex3           & 10  & 4   & 3   & 383.00        & True     & -             & -        & 0.25            & True     \\
		ex4           & 10  & 3   & 2   & 1.16          & False    & 0.63          & False    & 0.02            & True     \\
		ex5           & 10  & 3   & 2   & 2.41          & False    & 1.32          & False    & 0.03            & True     \\
		ex6           & 10  & 4   & 3   & 943.48        & True     & -             & -        & 0.35            & True     \\
		ex7           & 10  & 4   & 3   & 582.42        & True     & -             & -        & 0.82            & True     \\ \hline
	\end{tabular}
\end{table}

Here we also verified whether the flatness conditions \cref{equ:rank condition} and \cref{equ:multi-rank} hold for solutions computed from $\xi_t$ and $\xi_t^{isp}$, respectively. The results are provided in \cref{tab:cp} denoted by `Flat(1)' and `Flat(2)'. Note that since we used SVD to obtain the numerical rank of matrices, it might be affected by the tolerance or the random generated SOS polynomial in the objective function. As mentioned in \cite{henrion2005detecting}, one may try to extract the solution even if flatness conditions does not hold. For examples in \cref{tab:cp} satisfying the flatness conditions, we construct completely positive tensors $\mathcal{A}_{rec}$ using the extracted solution of the ideal-sparse model. In all cases but ex6, when compare to the original tensor $\mathcal{A}$, we obtain $\|\mathcal{A}_{rec}-\mathcal{A}\|_1\le 10^{-5}$. For ex6, we can obtain the same accuracy if the relaxation level $t$ is increased to 4.

Nonetheless, we remark that the decomposition factors extracted from the ideal-sparse moment relaxations may consist of repeated vectors, since the support sets of different maximal cliques may have common nodes. For instance, Fan et al. \cite{fan2017semidefinite} has given a completely positive decomposition of length 20 for ex3 as presented in \cref{tab:fan}. The v-row is the non-zero value of each decomposition vector and the p-row is the position of non-zero value. The decomposition obtained from our ideal-dense model is given in \cref{tab:decomposition}. The mc-row tells which clique the vector is extracted from. Moreover, both decompositions are consistent with \cref{lem:clique decomposition}.
\begin{table}[H]\label{tab:fan}
	\caption{The completely positive decomposition for ex3 given by Fan et al. \cite{fan2017semidefinite}}
	\begin{center}
		\setlength{\tabcolsep}{1mm}
	\begin{tabular}{|l|l|l|l|l|l|l|l|l|l|l|}
		\hline
		v & 1.0000  & 1.0000   & 1.0000   & 1.0000  & 1.0000  & 0.7071 & 0.7071 & 0.7071 & 1.0000     & 1.0000     \\ \hline
		p & 6      & 4      & 7      & 2      & 5 & 2,4    & 6,10   & 4,8    & 10      & 8      \\ \hline
		v & 0.7071 & 0.7071 & 0.7071 & 0.5774 & 1.0000  & 0.7071 & 0.5774 & 0.7071 & 0.5000  & 0.5774 \\ \hline
		p & 1,10   & 4,9    & 2,8    & 6,7,9  & 9 & 2,9    & 5,7,9  & 8,9    & 2,4,8,9 & 8,9,10 \\ \hline
	\end{tabular}
\end{center}
\end{table}
\begin{table}[H]\label{tab:decomposition}
	\caption{The completely positive decomposition for ex3 given by the ideal-sparse model}
	\begin{center}
		\setlength{\tabcolsep}{0.5mm}
	\begin{tabular}{|l|l|l|l|l|l|l|l|l|l|}
		\hline
		mc & \{1,10\}    & \{1,10\}    & \{2,4,8,9\} & \{2,4,8,9\} & \{2,4,8,9\} & \{2,4,8,9\} & \{2,4,8,9\} & \{2,4,8,9\} & \{2,4,8,9\} \\ \hline
		v  & 0.7071        & 1.0000        & 0.7071        & 0.7071        & 0.7071        & 1.0000        & 0.7071        & 1.0000        & 0.7071        \\ \hline
		p  & 1,10          & 10            & 2,4           & 2,9           & 4,9           & 4             & 2,8           & 2             & 8,9           \\ \hline
		mc & \{2,4,8,9\} & \{2,4,8,9\} & \{2,4,8,9\} & \{2,4,8,9\} & \{5,7,9\}   & \{5,7,9\}   & \{5,7,9\}   & \{5,7,9\}   & \{6,7,9\}   \\ \hline
		v  & 0.7071        & 1.0000        & 1.0000        & 0.5000        & 0.5774        & 1.0000        & 1.0000        & 1.0000        & 1.0000        \\ \hline
		p  & 4,8           & 8             & 9             & 2,4,8,9       & 5,7,9         & 9             & 5             & 7             & 9             \\ \hline
		mc & \{6,7,9\}   & \{6,7,9\}   & \{6,7,9\}   & \{6,10\}    & \{6,10\}    & \{6,10\}    & \{8,9,10\}  & \{8,9,10\}  & \{8,9,10\}  \\ \hline
		v  & 0.5774        & 1.0000        & 1.0000        & 1.0000        & 1.0000        & 0.7071        & 1.0000        & 1.0000        & 1.0000        \\ \hline
		p  & 6,7,9         & 6             & 7             & 6             & 10            & 6,10          & 10            & 9             & 8             \\ \hline
		mc & \{8,9,10\}  & \{8,9,10\}  &               &               &               &               &               &               &               \\ \hline
		v  & 0.7071        & 0.5774        &               &               &               &               &               &               &               \\ \hline
		p  & 8,9           & 8,9,10        &               &               &               &               &               &               &               \\ \hline
	\end{tabular}
\end{center}
\end{table}

\subsection{Comparison with TSSOS}
Korda et al. \cite{korda2024exploiting} have discussed the relationship between the ideal-sparsity and correlative sparsity, while the term sparsity is not completely clear for a given generalized moment problem. Here we conduct come preliminary experiments to show that the ideal-sparse generalized moment problem reformulation indeed brings unique benefits that differ from exploiting correlative and term sparsity in the dual SOS relaxations \cref{equ:dense dual} of the dense moment relaxations \cref{equ:dense relax}. 

The computation times of $\xi_t^{isp}$ solving the ideal-sparse moment relaxations and TSSOS solving the dual SOS relaxations \cref{equ:dense dual} are presented in \cref{tab:tssos}. `Times (s)' is the total computation time including the modeling process. In our tests, TSSOS is used with option {TS=``block"} and {CS=``MF"}, see \cite{wang2021tssos,wang2022cs} for more details. From \cref{tab:tssos}, we can observe that the ideal-sparse reformulation is indeed highly efficient in terms of the completely positive tensor decomposition problem. We leave it for further research to combine the three kinds of sparsities and solve the dual problem \cref{equ:sparse dual} of the ideal-sparse moment relaxations \cref{equ:sparse relax} using TSSOS, which we will not discuss more within this article.
\begin{table}[H]\label{tab:tssos}
	\caption{Comparison with TSSOS}
	\begin{center}
	\begin{tabular}{lllllll}
		\hline
		$\mathcal{A}$ & n  & d & \multicolumn{2}{l}{TSSOS}  & \multicolumn{2}{l}{$\xi_t^{isp}$} \\ \cline{4-7} 
		&    &   & SDP (s)    & Times (s) & SDP (s)        & Times (s)    \\ \hline
		ex1           & 3  & 3 & \textless{}0.01 & 0.01     & \textless{}0.01     & 0.01        \\
		ex2           & 10 & 3 & 0.40            & 0.45     & 0.03                & 0.05        \\
		ex3           & 10 & 4 & 56.09           & 57.50    & 0.25                & 1.11        \\
		ex4           & 10 & 3 & 0.26            & 0.33     & 0.02                & 0.04        \\
		ex5           & 10 & 3 & 0.69            & 0.76     & 0.03                & 0.06        \\
		ex6           & 10 & 4 & 58.38           & 59.94    & 0.35                & 1.42        \\
		ex7           & 10 & 4 & 59.19           & 60.86    & 0.82                & 2.69        \\
		non\_ex1      & 11 & 3 & 0.29            & 0.37     & \textless{}0.01     & 0.02        \\
		non\_ex2      & 8  & 5 & 1.76            & 2.12     & \textless{}0.01     & 0.03        \\ \hline
	\end{tabular}
\end{center}
\end{table}

\section{Conclusions}\label{sec:conclusion}
We consider completely positive tensor decomposition problems with ideal-sparsity. The maximal cliques of a support multi-hypergraph associated with completely positive tensors is introduced to construct the ideal-sparse generalized moment problem reformulation. We also propose an algorithm to generate all such maximal cliques and give a necessary condition for tensors to be completely positive. The ideal-sparse moment relaxations are applied to solve the new reformulation, for which asymptotic convergence of the ideal-sparse moment hierarchy is provided under some general assumptions. Moreover, the finite convergence is guranteed if the dense moment hierarchy is finite convergent. Numerical examples are presented to show the superiority of our new approach.

There are several further reseasrch directions that are left. First, there are various improvement methods for the maximum clique algorithm of graph in the literature. Can they be applied to the maximal cliques generation algorithm for multi-hypergraphs? Second, as mentioned in \cref{thm:maximal sets}, our proposed \cref{alg:clique} indeed finds the maximal subsets that do not contain any set $\{i_1,\cdots,i_m\}$ with $x_{i_1}\cdots x_{i_m}=0$. One would naturally expect to find other applications of ideals generated by a set of monomials $\prod_{i\in S} x_i(S\in\mathcal{S})$ where $\mathcal{S}$ is a given collection of subsets of $[n]$. In particular, it is worthwhile to exploit the ideal-sparsity in non-negative decomposition problems of non-symmetric tensors \cite{cichocki2009nonnegative} and comlex completely positive tensors \cite{zhou2020completely}. Third, combining the correlative and term sparsity with the dual problem of ideal-sparse moment relaxations may lead to further improvements. We leave it for future work.

\section*{Acknowledgments} The authors would like to thank Dr. Jie Wang for his help with utilizing TSSOS.

\bibliographystyle{siamplain}
\bibliography{references}
\end{document}